\RequirePackage{fix-cm}
\documentclass[smallextended]{svjour3}       
\smartqed  
\usepackage{graphicx}
\usepackage{amsfonts}

\spnewtheorem{assumption}{Assumption}{\bf}{\it}
\journalname{Journal of Optimization Theory and Applications}
\begin{document}

\title{A Framework of Conjugate Direction Methods for \\ Symmetric Linear Systems in Optimization\thanks{This author thanks the Italian national research program `RITMARE', by {\em CNR-INSEAN, National Research Council-Maritime Research Centre}, for the support received.}
}

\titlerunning{A Framework of Conjugate Direction Methods}        

\author{Fasano Giovanni}

\authorrunning{Fasano Giovanni} 

\institute{Fasano Giovanni \at
              Department of Management, University Ca'Foscari of Venice \\
              Tel.: +39-041234-6922\\
              Fax: +39-041234-7444\\
              \email{fasano@unive.it} \\
             \emph{Present address:} S.Giobbe, Cannaregio 873, 30121 Venice, ITALY
}

\date{Received: date / Accepted: date}

\maketitle

\begin{abstract}
In this paper we introduce a parameter dependent class of Krylov-based methods, namely $CD$, for the solution of
symmetric linear systems. We give evidence that in our proposal we generate sequences of conjugate directions, extending some properties of the standard Conjugate Gradient (CG) method, in order to preserve the conjugacy. For specific values of the parameters in our framework we obtain schemes equivalent to both the CG and the scaled-CG. We also prove the finite convergence of the algorithms in $CD$, and we provide some error analysis. Finally, preconditioning is introduced for $CD$, and we show that standard error bounds for the preconditioned CG also hold for the preconditioned $CD$.
\end{abstract}

\section{Introduction}
\label{sec_intro_0}
The solution of symmetric linear systems arises in a wide range of real applications \cite{A96,GV89,S03}, and has been carefully issued in the last 50 years, due to the increasing demand of fast and reliable solvers. {\em Illconditioning} and {\em large number of unknowns} are among the most challenging issues which may harmfully affect the solution of linear systems, in several frameworks where either structured or unstructured coefficient matrices are considered \cite{A96,H96,SVdV00}.

The latter facts have required the introduction of a considerable number of techniques, specifically aimed at tackling classes of linear systems with appointed pathologies \cite{SVdV00,GS92}. We remark that the structure of the coefficient matrix may be essential for the success of the solution methods, both in numerical analysis and optimization contexts. As an example, PDEs and PDE-constrained optimization provide two specific frameworks, where sequences of linear systems often claim for specialized and robust methods, in order to give reliable solutions.

In this paper we focus on iterative Krylov-based methods for the solution of symmetric linear systems, arising in both numerical analysis and optimization contexts. The theory detailed in the paper is not limited to consider large scale linear systems; however, since Krylov-based methods have proved their efficiency when the scale is large, without loss of generality we will implicitly assume the latter fact.

The accurate study and assessment of methods for the solution of linear systems is naturally expected from the community  of people working on numerical analysis. That is due to their expertise and great sensibility
to theoretical issues, rather than to practical algorithms implementation or software developments. This has raised a consistent literature, including manuals and textbooks, where the analysis of solution techniques for linear systems has become a keynote subject, and where essential achievements have given strong guidelines to theoreticians and practitioners from optimization \cite{H96}.

We address here a parameter dependent class of CG-based methods, which can equivalently reduce to the CG for a suitable choice of the parameters. We firmly claim that our proposal is not primarily intended to provide an efficient alternative to the CG. On the contrary, we mainly detail a general framework of iterative methods, inspired by polarity for quadratic hypersurfaces, and based on the generation of conjugate directions. The algorithms in our class, thanks to the parameters in the scheme, may possibly keep under control the conjugacy loss among directions, which is often caused by finite precision in the computation. The paper is not intended to report also a significant numerical experience. Indeed, we think that there are not yet clear rules on the parameters of our proposal, for assessing efficient algorithms. Similarly, we have not currently indications that methods in our proposal can outperform the CG. On this guideline, in a separate paper we will carry on selective numerical tests, considering both symmetric linear systems from numerical analysis and optimization. We further prove that preconditioning can be introduced for the class of methods we propose, as a natural extension of the preconditioned CG (see also \cite{GV89}).

As regards the symbols used in this paper, we indicate with $\lambda_m(A)$ and $\lambda_M(A)$ the
smallest/largest eigenvalue of the positive definite matrix $A$;
moreover $\|v\|^2_{A}= v^TAv$, where $A$ is a positive definite
real matrix. $R(A)$ is the {\em range} of matrix $A$ and $A^+$ is the Moore-Penrose pseudoinverse of matrix $A$. With $Pr_{C}(v)$ we represent the orthogonal projection of vector $v\in \mathbb{R}$ onto the convex set $C \subseteq \mathbb{R}$. Finally, the symbol ${\cal K}_i(b,A)$ indicates the Krylov subspace ${\rm span}\{b,Ab,A^2b, \dots, A^ib\}$ of dimension $i+1$. All the other symbols in the paper follow a standard
notation.

Sect.~\ref{sec2A}  briefly reviews both the CG and the Lanczos process, as Krylov-subspace methods, in order to highlight promising aspects to investigate in our proposal. Sect.~\ref{sec:2b} details some relevant applications of conjugate directions in optimization frameworks, motivating our interest for possible extensions of the CG. In Sects.~ \ref{sec_prop} and \ref{sec:further_prop} we describe our class of methods and some related properties. In Sects.~ \ref{sec:relaz_CG-CG_2step} and \ref{sec:scaled} we show that the CG and the scaled-CG may be equivalently obtained as particular members of our class. Then, Sects.~ \ref{sec:matr_fact} and \ref{sec:conjugacy_loss} contain further properties of the class of methods we propose. Finally, Sect.~\ref{sec:prec_CG_2step} analyzes the preconditioned version of our proposal, and a section of Conclusions completes the paper, including some numerical results.

\section{The CG Method and the Lanczos Process}
\label{sec2A}
In this section we comment the method in Table \ref{tab_CG}, and we focus on the relation between the CG and the Lanczos process, as Krylov-subspace methods. In particular, the Lanczos process namely does not generate conjugate directions; however, though our proposal relies on generalizing the CG, it shares some aspects with the Lanczos iteration, too.
\\
As we said, the CG is commonly used to iteratively solving the linear system
\begin{table}
\caption{The CG algorithm for solving (\ref{system}).}
\begin{center}
\fbox{$\!\!\!\!\!$
\begin{tabular}{l}
    \   \\
\multicolumn{1}{c}{\bf The Conjugate Gradient (CG) method}  \\
    \   \\
{\bf Step $0$:} $\quad\!$ Set $k=0$, \ $y_0 \in \mathbb{R}$, \ $r_{0}:= b-A y_0$.  \\
\hspace*{1cm} $\quad$ If $r_{0}=0$, then STOP. Else, set $p_0 := r_0$; \ $k=k+1$.   \\
\hspace*{1cm} $\quad$ Set $p_{-1}=0$ and $\beta_{-1}=0$.  \\
{\bf Step $k$:} $\quad\!$ Compute $\alpha_{k-1}:= r_{k-1}^{T}p_{k-1}/p_{k-1}^{T}Ap_{k-1}$,   \\
\hspace*{1.0cm} $\quad$ $y_{k}:= y_{k-1} + \alpha_{k-1}p_{k-1}$, \ $r_{k}:= r_{k-1} - \alpha_{k-1}Ap_{k-1}$. \\ \hspace*{1.0cm} $\quad$ If $r_{k}=0$, then STOP. Else, set \\
\hspace*{1.0cm} $\quad$ \ \ -- \ $\beta_{k-1}:= \|r_{k}\|^2/\|r_{k-1}\|^2$, $p_{k}:= r_{k} + \beta_{k-1}p_{k-1}$      \\
\hspace*{1.0cm} $\quad$ \ \ -- \ (or equivalently set $p_{k}:= -\alpha_{k-1}Ap_{k-1} + (1+ \beta_{k-1})p_{k-1} -\beta_{k-2}p_{k-2}$)   \\
\hspace*{1.0cm} $\quad$ Set $k = k+1$, go to {\bf Step $k$}.
\end{tabular}
$\!\!\!\!\!$}
\end{center}
\label{tab_CG}
\end{table}
%
%
\begin{equation}
   Ay = b,     \label{system}
\end{equation}
where $A \in \mathbb{R}^{n \times n}$ is symmetric {\em positive
definite} and $b \in \mathbb{R}^n$. Observe that the CG is quite often
applied to a {\em preconditioned} version of the linear system
(\ref{system}), i.e. ${\cal M}Ay={\cal M}b$, where ${\cal M} \succ
0$ is the preconditioner \cite{G97}. Though the theory for the CG
requires $A$ to be positive definite, in several practical
applications it is successfully used when $A$ is indefinite, too
\cite{H80,Nash2000}. At Step $k$ the CG generates the pair
of vectors $r_k$ ({\em residual}) and $p_k$ ({\em search
direction}) such that \cite{GV89}
\begin{eqnarray}
        & & {\rm orthogonality \ \ property:}  \qquad \ r_i^Tr_j = 0,  \qquad\qquad \ \ \ 0 \leq i \not = j \leq k,   \label{rela_1a}  \\
       & & {\rm conjugacy \ \ property:}  \quad\qquad \ \, p_i^TAp_j = 0, \qquad\qquad 0 \leq i \not = j \leq k.  \label{rela_2a}
\end{eqnarray}
Moreover, finite convergence holds, i.e. $Ay_h=b$ for some $h \leq
n$. Relations (\ref{rela_1a}) yield the Ritz-Galerkin condition
$r_{k} \perp {\cal K}_{k-1}(r_0,A)$, where
    $$ {\cal K}_{k-1}(r_0,A) := {\rm span}\{b,Ab,A^2b, \ldots, A^{k-1}b\} \equiv {\rm span}\{r_0, \ldots, r_{k-1}\}.$$
Furthermore, the direction $p_{k}$ is computed at Step $k$
imposing the conjugacy condition $p_{k}^TAp_{k-1}=0$. It can be
easily proved that the latter equality implicitly satisfies
relations (\ref{rela_2a}), with $p_0, \ldots, p_k$ linearly
independent. We remark that on practical problems, due to {\em
finite precision} and {\em roundoff} in the computation of the
sequences $\{p_k\}$ and $\{r_k\}$, when $|i-j|$ is large relations
(\ref{rela_1a})-(\ref{rela_2a}) may fail. Thus, in the practical
implementation of the CG some theoretical properties may not be
satisfied, and in particular when $|i-j|$ increases the conjugacy properties
(\ref{rela_2a}) may progressively be lost. As detailed in
\cite{CGT00,F05,GLL89,MN00}  the latter fact may have dramatic consequences also in
optimization frameworks (see also Sect.~\ref{sec:2b} for details). To our purposes
we note that in Table \ref{tab_CG}, at Step $k$
of the CG, the direction $p_k$ is usually computed as
\begin{equation}
     p_k := r_k +\beta_{k-1}p_{k-1},         \label{eq:standard_CG}
\end{equation}
but an equivalent expression is (see also Theorem 5.4 in \cite{HS52})
\begin{equation}
     p_k := -\alpha_{k-1}Ap_{k-1} +(1+\beta_{k-1})p_{k-1} - \beta_{k-2}p_{k-2}, \label{eq:altern_CG}
\end{equation}
which we would like to generalize in our proposal. Note also that in exact arithmetics the property (\ref{rela_2a}) is iteratively fulfilled by both (\ref{eq:standard_CG}) and (\ref{eq:altern_CG}).

The Lanczos process (and its preconditioned version) is another
Krylov-based method, widely used to tridiagonalize the matrix $A$ in (\ref{system}). Unlike
the CG method, here the matrix $A$ may be possibly indefinite, and
the overall method is slightly more expensive than the CG, since further computation is necessary to solve the resulting tridiagonal system.
Similarly to the CG, the Lanczos process generates at Step $k$ the
sequence $\{u_k\}$ ({\em Lanczos vectors}) which satisfies
    $$ {\rm orthogonality \ \ property:} \qquad u_i^Tu_j = 0,  \qquad 0 \leq i \not = j \leq k, $$
and yields finite convergence in at most $n$ steps. However, unlike the CG the Lanczos process is not explicitly inspired by polarity, in order to generate the orthogonal vectors. We recall that the CG and the Lanczos process are
3-term recurrence methods, in other words, for $k \geq 1$
    $$ \begin{array}{l}
        p_{k+1} \in {\rm span} \{Ap_k, p_k, p_{k-1}\}, \qquad {\rm for \ the \ CG}   \\
            \    \\
        u_{k+1} \in {\rm span} \{Au_k, u_k, u_{k-1}\}, \qquad {\rm for \ the \ Lanczos \ process}.
       \end{array} $$
When $A$ is positive definite, a full theoretical correspondence
between the sequence $\{r_k\}$ of the CG and the sequence $\{u_k\}$ of the Lanczos process may be fruitfully used in optimization problems (see also \cite{CGT00,F07,S83}),
being
    $$ u_k= s_k \frac{r_k}{\|r_k\|}, \qquad \qquad s_k \in \{-1,+1\}. $$

The class $CD$ proposed in this paper provides a framework, which encompasses the CG and to some extent resembles the
Lanczos iteration, since a 3-term recurrence is exploited. In particular, the $CD$ generates both
conjugate directions (as the CG) and orthogonal residuals (as the
CG and the Lanczos process). Moreover, similarly to the CG, the $CD$ yields a
3-term recurrence with respect to conjugate directions. As we remarked,
our proposal draws its inspiration from the idea of possibly attenuating the conjugacy loss of the CG, which may occur in (\ref{rela_2a}) when $|i-j|$ is large.

\section{Conjugate Directions for Optimization Frameworks}
\label{sec:2b} Optimization frameworks offer plenty of symmetric
linear systems where CG-based methods are often specifically preferable
with respect to other solvers. Here we justify this
statement by briefly describing the potential use of conjugate
directions within truncated Newton schemes. The latter methods
strongly prove their efficiency when applied to large scale
problems, where they rely on the proper computation of search
directions, as well as truncation rules (see \cite{NS90}).
\\
As regards the computation of search directions, suppose at the
outer iteration $h$ of the truncated scheme we perform $m$ steps of the CG, in order to compute the approximate solution $d_h^m$ to the linear system (Newton's equation)
    $$ \nabla^2f(z_h) d = -\nabla f(z_h).$$
When $z_h$ is close enough to the solution $z^\ast$ (minimum point) then possibly $\nabla^2 f(z_h) \succ 0$.  Thus, the conjugate directions
$p_1, \ldots, p_m$ and the coefficients $\alpha_1, \ldots,
\alpha_m$ are generated as in Table \ref{tab_CG}, so that the
following vectors can be formed
\begin{equation}
    \begin{array}{lcl}
        \displaystyle d_h^m = \sum_{i=1}^m \alpha_i p_i,  & &  \\
        \displaystyle d_h^P = \sum_{i \in I_h^P} \alpha_i p_i, & \  &
        I_h^P = \left\{i \in \{1, \ldots, m\} : \ \ p_i^T \nabla^2 f(z_h)p_i > 0 \right\}, \\
        \displaystyle d_h^N = \sum_{i \in I_h^N} \alpha_i p_i, & \  &
        I_h^N = \left\{i \in \{1, \ldots, m\} : \ \ p_i^T \nabla^2 f(z_h)p_i < 0 \right\}, \\
        \displaystyle s_h= \frac{p_\ell}{\|r_\ell\|}, & \  &
        \displaystyle \ell = {\arg\min}_{i \in \{1, \ldots,m\}} \left\{ \frac{p_i^T \nabla^2 f(z_h)p_i}{\|r_i\|^2}: \ \
        p_i^T \nabla^2 f(z_h)p_i < 0 \right\}.
    \end{array}   \label{directions}
\end{equation}
Observe that $d_h^m$ approximates in some sense Newton's direction at the outer
iteration $h$, and as described in \cite{F05,GLL89,FaRo07,GLRT00}
the vectors $d_h^m$, $d_h^P$ and $d_h^N$ can be used/combined to
provide fruitful search directions to the optimization framework. Moreover, $d_h^N$ and $s_h$ are suitably used/combined
to compute a so called {\em negative curvature direction}
`$s_h^m$', which can possibly force second order convergence for the
overall truncated optimization scheme (see \cite{FaRo07} for details). The
conjugacy property is essential for computing the vectors
(\ref{directions}). i.e. to design efficient truncated Newton
methods. Thus, introducing CG-based schemes which deflate conjugacy loss might be of great importance.

On the other hand, at the outer iteration $h$ effective truncation
rules typically attempt to assess the parameter $m$ in (\ref{directions}), as
described in \cite{NS90,FaLu09,NW00}. I.e., they monitor the
decrease of the quadratic local model
    $$ Q_h(d_h^m) := f(z_h) + \nabla f(z_h)^T (d_h^m) + \frac{1}{2}(d_h^m)^T \nabla^2 f(z_h) (d_h^m) $$
when $\nabla^2f(z_h) \succ 0$, so that the parameter $m$ is chosen to satisfy some conditions, including
    $$ \frac{Q_h(d_h^m) - Q_h(d_h^{m-1})}{Q_h(d_h^m)/m} \leq \alpha, \qquad {\rm for \ some \ }\alpha \in \ ]0,1[.$$
Thus, again the correctness of conjugacy properties among the directions $p_1,
\ldots, p_m$, generated while solving Newton's equation, may be
essential both for an accurate solution of Newton's equation (which is a linear system) and to the overall efficiency of the truncated optimization method.

\section{Our Proposal: the $CD$ Class}
\label{sec_prop}
Before introducing our proposal for a new general framework of CG-based algorithms, we consider here some additional  motivations for using the CG.
The careful use of the latter theory is in our opinion a launching pad for possible extensions of the CG. On this guideline, recalling the contents in Sect.~\ref{sec:2b}, now we summarize some critical aspects of the CG:
\begin{enumerate}
    \item the CG works iteratively and at any iteration the overall computational effort is only $O(n^2)$ (since the CG is a Krylov-subspace method);
    \item the conjugate directions generated by the CG are linearly independent, so that at most $n$ iterations are necessary to address the solution;
    \item the current conjugate direction $p_{k+1}$ is computed by simply imposing the conjugacy with respect to the direction $p_k$ (computed) in the previous iteration. This automatically yields that $p_{k+1}^TAp_i = 0$, for any $i \leq k$, too.
\end{enumerate}
As a matter of fact, for the design of possible general frameworks including CG-based methods, the items 1. and 2. are essential in order to respectively control the {\em computational effort} and ensure the {\em finite convergence}.

On the other hand, altering the item 3. might be harmless for the overall iterative process, and might possibly yield some fruitful generalizations. That is indeed the case of our proposal, where the item 3. is modified with respect to the CG. The latter modification depends on a parameter which is user/problem-dependent, and may be set in order to further compensate or correct the conjugacy loss among directions, due to roundoff and finite precision.

We sketch in Table~\ref{tab_CG-2step} our new
CG-based class of algorithms, namely $CD$.
\begin{table}
\caption{The parameter dependent class $CD$ of CG-based algorithms for solving (\ref{system}).}
\begin{center}
\fbox{
\begin{tabular}{l}
    \   \\
\multicolumn{1}{c}{\bf The $CD$ class}  \\
    \   \\
{\bf Step $0$:} $\quad\!$ Set $k=0$, $y_0 \in \mathbb{R}^n$, $r_0:=b-Ay_0$, $\gamma_0 \in \mathbb{R}\setminus\{0\}$. \\
\hspace*{1cm} $\quad$ If $r_0 = 0$, then STOP. Else, set $p_0 := r_0$, $k=k+1$.  \\
\hspace*{1cm} $\quad$ Compute $a_0:=r_0^Tp_0 / p_0^TAp_0$, \\
\hspace*{1cm} $\quad$ $y_1 := y_0 + a_0 p_0$, \ $r_1 := r_0 - a_0 Ap_0$.  \\
\hspace*{1cm} $\quad$ If $r_1 = 0$, then STOP. Else, set $\sigma_0:= \gamma_0\|Ap_0\|^2 / p_0^TAp_0$, \\
\hspace*{1cm} $\quad$ $p_1 := \gamma_0Ap_0 -\sigma_0p_0$, \  $k = k+1$. \\
{\bf Step $k$:}  $\quad\!$ Compute $a_{k-1}:= r_{k-1}^Tp_{k-1} / p_{k-1}^TAp_{k-1}$, \\
\hspace*{1cm} $\quad$ $y_k := y_{k-1} + a_{k-1} p_{k-1}$, \ $r_k := r_{k-1} - a_{k-1}Ap_{k-1}$. \\
\hspace*{1cm} $\quad$ If $r_k = 0$, then STOP. Else, set $\sigma_{k-1}:= \gamma_{k-1}\frac{\|Ap_{k-1}\|^2}{p_{k-1}^TAp_{k-1}}$, \\
\hspace*{1cm} $\quad$ $\omega_{k-1}:=\gamma_{k-1}\frac{(Ap_{k-1})^TAp_{k-2}}{p_{k-2}^TAp_{k-2}} = \frac{\gamma_{k-1}}{\gamma_{k-2}} \frac{p_{k-1}^TAp_{k-1}}{p_{k-2}^TAp_{k-2}}$, \ $\gamma_{k-1} \in \mathbb{R}\setminus\{0\}$  \\
\hspace*{1cm} $\quad$ $p_k := \gamma_{k-1}Ap_{k-1} - \sigma_{k-1}p_{k-1} - \omega_{k-1}p_{k-2}$, \ $k = k+1$. \\
\hspace*{1cm} $\quad$ Go to {\bf Step $k$}.
\end{tabular}
}
\end{center}
\label{tab_CG-2step}
\end{table}
The computation of the direction $p_k$ at Step $k$ reveals the main
difference between the CG and $CD$. In
particular, in Table~\ref{tab_CG-2step} the pair of coefficients
$\sigma_{k-1}$ and $\omega_{k-1}$ is computed so that {\em explicitly}\footnote{A further generalization might be obtained computing $\sigma_{k-1}$ and $\omega_{k-1}$ so that
\begin{equation}
       \left\{ \begin{array}{l}
                    p_k^TA(\gamma_{k-1}Ap_{k-1} - \sigma_{k-1}p_{k-1}) \ = \ 0 \\
                        \       \\
                    p_k^TAp_{k-2} \ = \ 0.
                \end{array} \right. \label{two_star_alternative}
\end{equation}
}
\begin{equation}
                \begin{array}{l}
                    p_k^TAp_{k-1} \ = \ 0   \\
                     \       \\
                    p_k^TAp_{k-2} \ = \ 0,
                \end{array}  \label{two_star}
\end{equation}
i.e. in Cartesian coordinates the conjugacy between the direction $p_k$ and both the
directions $p_{k-1}$ and $p_{k-2}$ is directly imposed, as specified by (\ref{rela_2a}). As detailed in
Sect.~\ref{sec2A}, imposing the double condition (\ref{two_star}) allows to possibly recover the conjugacy loss in the sequence $\{p_i\}$.

On the other hand, the residual $r_k$ at Step $k$ of Table \ref{tab_CG-2step} is computed by imposing the orthogonality condition $r_k^Tp_{k-1}=0$, as in the standard CG. The resulting method is {\em evidently a bit more expensive than the CG}, requiring one additional inner product per step, as long as an additional scalar to compute and an additional $n$-vector to store.  From Table \ref{tab_CG-2step} it is also evident that $CD$ provides a 3-term recurrence
with respect to the conjugate directions.

In addition, observe that the residual $r_k$ is computed at Step $k$ of $CD$ only to check for the stopping condition, and is not directly involved in the computation of $p_k$. Hereafter in this section we briefly summarize the basic properties of the class $CD$.
\begin{assumption}
\label{assu_0} The matrix $A$ in (\ref{system}) is symmetric
positive definite. Moreover, the sequence $\{\gamma_k\}$ in Table \ref{tab_CG-2step} is such that $\gamma_k \not = 0$, for any $k \geq 0$.
\end{assumption}

Note that as for the CG, the Assumption \ref{assu_0} is required
for theoretical reasons. However, the $CD$ class may in principle
be used also in several cases when $A$ is indefinite, provided that $p_k^TAp_k \neq 0$, for any $k \geq 0$.

\begin{lemma}
\label{lemma0} Let Assumption~\ref{assu_0} hold. At Step $k$ of
the {\rm $CD$} class, with $k \geq 0$, we have
\begin{equation}
    Ap_j \ \in \ {\rm span} \left\{p_{j+1},p_j,p_{\max\{0,j-1\}} \right\}, \ \ \ \ \ j \leq k.  \label{four_star}
\end{equation}
\end{lemma}
\begin{proof}
From the Step~0 relation (\ref{four_star}) holds for $j=0$. Then,
for $j=1, \ldots, k-1$ the Step $j+1$ of $CD$ directly
yields (\ref{four_star}). $\hfill \Box$
\end{proof}
\begin{theorem}
\label{theorem1}{\em \bf [Conjugacy]} Let Assumption~\ref{assu_0} hold. At Step $k$ of
the {\rm $CD$} class, with $k \geq 0$, the
directions \ $p_0,p_1, \ldots, p_k$ are mutually conjugate, i.e. $p_i^TAp_j=0$, with $0 \leq i \not = j \leq k$.
\end{theorem}
\begin{proof}
The statement holds for Step~0, as a consequence of the choice of the
coefficient $\sigma_0$. Suppose it holds for $k-1$; then, we
have for $j \leq k-1$
\begin{eqnarray*}
   p_k^TAp_j & = & \left(\gamma_{k-1}Ap_{k-1} -\sigma_{k-1}p_{k-1} -\omega_{k-1}p_{k-2} \right)^TAp_j \\
                 & = & (\gamma_{k-1}Ap_{k-1})^TAp_j -\sigma_{k-1}p_{k-1}^TAp_j -\omega_{k-1}p_{k-2}^TAp_j \ = \ 0.
\end{eqnarray*}
In particular, for $j=k-1$ and $j=k-2$ the choice of the
coefficients $\sigma_{k-1}$ and $\omega_{k-1}$, and the inductive
hypothesis, yield directly $p_k^TAp_{k-1}=p_k^TAp_{k-2}=0$. For $j
< k-2$, the inductive hypothesis and Lemma \ref{lemma0} again
yield the conjugacy property. $\hfill \Box$
\end{proof}
\begin{lemma}
\label{lemma2} Let Assumption~\ref{assu_0} hold. Given
the {\rm $CD$} class, we have for $k \geq 2$
   $$ (Ap_k)^T(Ap_i) \ = \ \left\{ \begin{array}{ccl}
                                     \|Ap_k\|^2, & \ \ \ \ \ & {\rm if} \ \ i=k,  \\
                                                 &           &      \\
                                      \frac{1}{\gamma_{k-1}}p_k^TAp_k, &           & {\rm if} \ \ i=k-1, \\
                                                 &           &       \\
                                      \emptyset, &           & {\rm if} \ \ i \leq k-2.
                                  \end{array}    \right. $$
\end{lemma}
\begin{proof}
The statement is a trivial consequence of Step $k$ of the {\rm $CD$}, Lemma \ref{lemma0} and
Theorem \ref{theorem1}. $\hfill \Box$
\end{proof}

Observe that from the previous lemma, a simplified expression for
the coefficient $\omega_{k-1}$, at Step $k$ of $CD$ is
available, inasmuch as
\begin{equation}
   \omega_{k-1} \ = \ \frac{\gamma_{k-1}}{\gamma_{k-2}} \cdot \frac{p_{k-1}^TAp_{k-1}}{p_{k-2}^TAp_{k-2}}.   \label{two_x}
\end{equation}
Relation (\ref{two_x}) has a remarkable importance: it avoids the
storage of the vector $Ap_{k-2}$ at Step $k$, requiring only the
storage of the quantity $p_{k-2}^TAp_{k-2}$. Also observe that
unlike the CG, the sequence $\{p_k\}$ in $CD$ is
computed independently of the sequence $\{r_k\}$. Moreover, as we said the
residual $r_k$ is simply computed at Step $k$ in order to check
the stopping condition for the algorithm.

The following result proves that the $CD$ class
recovers the main theoretical properties of the
standard CG.
\begin{theorem}{\em \bf [Orthogonality]}
\label{teo:properties}
Let Assumption~\ref{assu_0} hold. Let $r_{k+1} \not = 0$ at Step
$k+1$ of the {\rm $CD$} class, with $k \geq 0$.
Then, the directions $p_0, p_1, \ldots, p_k$ and the residuals
$r_0, r_1, \ldots, r_{k+1}$ satisfy
\begin{eqnarray}
   r_{k+1}^Tp_j \ = \ 0, \ \ \ \ \ \ \ \ \ \ j \leq k,  & & \label{rela_1} \\
        \                                               & & \nonumber      \\
   r_{k+1}^Tr_j \ = \ 0, \ \ \ \ \ \ \ \ \ \ j \leq k.  & & \label{rela_2}
\end{eqnarray}
\end{theorem}
\begin{proof}
From Step $k+1$ of $CD$ we have $r_{k+1}=r_k - a_k
Ap_k = r_j - \sum_{i=j}^k a_i Ap_i$, for any $j \leq k$. Then, from
Theorem \ref{theorem1} and the choice of coefficient $\alpha_j$ we
obtain
   $$ r_{k+1}^Tp_j \ =  \ \left( r_j - \sum_{i=j}^k a_i Ap_i \right)^Tp_j  \ = \
   r_j^Tp_j - \sum_{i=j}^k a_i p_i^TAp_j \ = \ 0, \ \ \ \ \ \ \ \ j \leq k,  $$
which proves (\ref{rela_1}). As regards relation (\ref{rela_2}),
for $k=0$ we obtain from the choice of $a_0$
  $$ r_1^Tr_0  \ =  \ r_1^Tp_0  \ =  \ 0.  $$
Then, assuming by induction that (\ref{rela_2}) holds for $k-1$,
we have
\begin{eqnarray*}
   r_{k+1}^Tr_j &  =  & \left( r_k - a_k Ap_k \right)^T r_j  \ =  \ \left( r_k - a_k Ap_k \right)^T \left( r_0 - \sum_{i=0}^{j-1} a_i Ap_i \right)   \\
                &  =  & r_k^Tr_0 - \sum_{i=0}^{j-1} a_i r_k^TAp_i - a_k p_k^TAr_0 + \sum_{i=0}^{j-1} a_i a_k (Ap_k)^TAp_i, \ \ \ \ \ \ \ \ \ j \leq
                k.
\end{eqnarray*}
The inductive hypothesis and Theorem \ref{theorem1} yield for $j \leq k$ (in the next relation when $i=0$ then $p_{i-1} \equiv 0$)
\begin{equation}
     r_{k+1}^Tr_j  \ =  \ - \sum_{i=0}^{j-1} \frac{a_i r_k^T}{\gamma_i} \left(p_{i+1}+\sigma_i p_i +
     \omega_ip_{i-1} \right) +  \sum_{i=0}^{j-1} a_i a_k (Ap_k)^TAp_i.
     \label{new_star}
\end{equation}
Therefore, if $j=k$ the relation (\ref{rela_1}) along with Lemma \ref{lemma2}
and the choice of $a_k$ yield
\begin{eqnarray*}
     r_{k+1}^Tr_k  &  =  & - \frac{a_{k-1}}{\gamma_{k-1}}r_k^Tp_k + \frac{a_{k-1} a_k}{\gamma_{k-1}} p_k^TAp_k  \ = \ 0.
\end{eqnarray*}
On the other hand, if $j<k$ in (\ref{new_star}),
 the inductive hypothesis, relation (\ref{rela_1}) and Lemma
\ref{lemma2} yield (\ref{rela_2}).
$\hfill \Box$
\end{proof}

Finally, we prove that likewise the CG, in at most $n$
iterations $CD$ determines the solution of the linear
system (\ref{system}), so that finite convergence holds.
\begin{lemma}{\em \bf [Finite convergence]}
\label{lem:finite_convergence}
Let Assumption~\ref{assu_0} hold. At Step $k$ of the $CD$ class, with $k \geq 0$,  the vectors $p_0, \ldots,
p_k$ are linearly independent. Moreover, in at most $n$ iterations
the {\rm $CD$} class computes the solution of the
linear system (\ref{system}), i.e. $Ay_h=b$, for some $h \leq n$.
\end{lemma}
\begin{proof}
The proof follows very standard guidelines (the reader may also refer
to \cite{M06}). Thus, by (\ref{rela_1}) an integer $m \leq n$ exists such that $r_m = b-Ay_m=0$. Then, if $y^\ast$ is the solution of (\ref{system}), we have
    $$ 0 = b-Ay_m = Ay^\ast - A \left[ y_0 + \sum_{i=0}^{m-1} a_i p_i \right]  \qquad \Longleftrightarrow \qquad y^\ast = y_0 + \sum_{i=0}^{m-1} a_i p_i. $$

$\hfill \Box$
\end{proof}

\begin{remark}
\label{rem:1}
{\em
Observe that there is the additional chance to replace the Step 0 in Table~\ref{tab_CG-2step}, with the following CG-like Step $0_b$
\begin{center}
\begin{tabular}{l}
{\bf Step $0_b$:} \ Set $k=0$, $y_0 \in \mathbb{R}^n$, $r_0:=b-Ay_0$. \\
\hspace*{1cm} $\quad$ If $r_0 = 0$, then STOP. Else, set $p_0:= r_0$, \ $k = k+1$.  \\
\hspace*{1cm} $\quad$ Compute $a_0:=r_0^Tp_0 / p_0^TAp_0$, \\
\hspace*{1cm} $\quad$ $y_1 := y_0 + a_0 p_0$, \ $r_1:= r_0 - a_0 Ap_0$.  \\
\hspace*{1cm} $\quad$ If $r_1 = 0$, then STOP. Else, set $\sigma_0:= -\|r_1\|^2 / \|r_0\|^2$, \ \\
\hspace*{1cm} $\quad$ $p_1 := r_1 +\sigma_0p_0$, \  $k = k+1$.
\end{tabular}
\end{center}
}
\end{remark}

\section{Further Properties for $CD$}
\label{sec:further_prop}
In this section we consider some properties of $CD$ which represent a natural extension of analogous properties of the CG. To this purpose we introduce the {\em error function}
\begin{equation}
     f(y) := \frac{1}{2}(y-y^\ast)^TA(y-y^\ast), \qquad {\rm with} \quad Ay^\ast =b, \label{equ:func_error}
\end{equation}
and the quadratic functional
\begin{equation}
     g(y) := \frac{1}{2}(y-y_i)^TA(y-y_i), \qquad {\rm with} \quad i \in \{1, \dots,m\},        \label{equ:g_func}
\end{equation}
which satisfy $f(y)\geq 0$, $g(y) \geq 0$, for any $y \in \mathbb{R}^n$, when $A \succeq 0$. Then, we have the following result, where we prove minimization properties of the error function $f(y)$ (see also Theorem 6.1 in \cite{HS52}) and $g(y)$ (see also \cite{P87}), along with the fact that $CD$ provides a suitable approximation of the inverse matrix $A^{-1}$, too.
\begin{theorem}
\label{prop:further_theor}{\em \bf [Further Properties]}
Consider the linear system (\ref{system}) with $A \succeq 0$, and the functions $f(y)$ and $g(y)$ in (\ref{equ:func_error})-(\ref{equ:g_func}). Assume that the $CD$ has performed $m+1$ iterations, with $m+1 \leq n$ and $Ay_{m+1}=b$. Let $\gamma_{i-1} \neq 0$ with $i \geq 1$. Then,
\begin{itemize}
    \item[$\bullet$] $\sigma_0$ minimizes $g(y)$ on the manifold $(y_1+\gamma_0Ap_0) + {\rm span} \{ p_0\}$,
    \item[$\bullet$] $\sigma_{i-1}$ and $\omega_{i-1}$, $i = 2,\dots,m$, minimize $g(y)$ on the two dimensional manifold $(y_i+\gamma_{i-1}Ap_{i-1}) + {\rm span} \{ p_{i-1}, p_{i-2}\}$.
\end{itemize}
Moreover,
\begin{equation}
     \displaystyle \quad f(y_i + a_ip_i) = f(y_i)- \left(\frac{\gamma_{i-1}}{a_{i-1}}\right)^2 \frac{\|r_i\|^4}{p_i^TAp_i}, \qquad i=1, \dots,m,            \label{equ:quad}
\end{equation}
and we have
\begin{equation}
    \displaystyle \left[A^+ - \sum_{i=0}^m \frac{p_ip_i^T}{p_i^TAp_i}\right] r_0 = 0, \qquad\qquad  {\rm for \ any \ } y_0 \in \mathbb{R}^n.        \label{equ:quad_2}
\end{equation}
\end{theorem}
\begin{proof}
Observe that for $i=1$, indicating in Table \ref{tab_CG-2step} \ $p_1 = \gamma_0 A p_0 + a p_0$, with $a \in \mathbb{R}$, by (\ref{equ:g_func})
    $$ g(y_2)=g(y_1+a_1p_1)= \frac{a_1^2}{2} (\gamma_0 A p_0 + a p_0)^TA(\gamma_0 A p_0 + a p_0)$$
and we have
    \[ 0 = \left.\frac{\partial g(y_2)}{\partial a}\right|_{a=a^\ast} =  a_1^2 p_0^TA(\gamma_0 A p_0 + a^\ast p_0) \quad \Longleftrightarrow \quad a^\ast = -\gamma_0 \frac{\|Ap_0\|^2}{p_0^Tap_0} = -\sigma_0.  \]
For $i \geq 2$, if we indicate in Table \ref{tab_CG-2step} \ $p_{i} = \gamma_{i-1}Ap_{i-1}+bp_{i-1}+cp_{i-2}$, with $b,c \in \mathbb{R}$, then by (\ref{equ:g_func})
\begin{eqnarray*}
    g(y_i+a_ip_i) & = \frac{a_i^2}{2}(\gamma_{i-1}Ap_{i-1}+bp_{i-1}+cp_{i-2})^TA(\gamma_{i-1}Ap_{i-1}+bp_{i-1}+cp_{i-2})
\end{eqnarray*}
and by Assumption~\ref{assu_0}, after some computation, the equalities
    $$ \left\{  \displaystyle
       \begin{array}{l}
            \displaystyle \left.\frac{\partial g(y_{i+1})}{\partial b}\right|_{b=b^\ast  , \ c=c^\ast} =
            \left. \frac{\partial g(y_i+a_ip_i)}{\partial b}\right|_{b=b^\ast, \ c=c^\ast} =
             0  \\
                \       \\
            \displaystyle \left.\frac{\partial g(y_{i+1})}{\partial c}\right|_{b=b^\ast  , \ c=c^\ast} =
            \left. \frac{\partial g(y_i+a_ip_i)}{\partial c}\right|_{b=b^\ast, \ c=c^\ast} =
            0
       \end{array} \right.$$
imply the unique solution
\begin{equation}
     \left\{  \displaystyle
       \begin{array}{l}
            \displaystyle b^\ast = - \gamma_{i-1} \frac{\|Ap_{i-1}\|^2}{p_{i-1}^TAp_{i-1}} = -\sigma_{i-1}  \\
                \       \\
            \displaystyle c^\ast = - \gamma_{i-1} \frac{(Ap_{i-1})^T(Ap_{i-2})}{p_{i-2}^TAp_{i-2}} = - \frac{\gamma_{i-1}}{\gamma_{i-2}} \frac{p_{i-1}^TAp_{i-1}}{p_{i-2}^TAp_{i-2}} = -\omega_{i-1}.
       \end{array} \right.      \label{equ:optimal_2dim_mani}
\end{equation}
As regards (\ref{equ:quad}), from Table \ref{tab_CG-2step} we have that for any $i \geq 1$
\begin{eqnarray}
    f(y_i + a_ip_i) & = & f(y_i) + a_i (y_i-y^\ast)^TAp_i + \frac{1}{2}a_i^2 p_i^TAp_i  \nonumber \\
                    & = & f(y_i) - a_i r_i^Tp_i + \frac{1}{2}a_i^2 p_i^TAp_i  \nonumber \\
                    & = & f(y_i) - \frac{1}{2}\frac{(r_i^Tp_i)^2}{p_i^TAp_i}. \label{eq:reduction_error}
\end{eqnarray}
Now, since $r_i=r_{i-1}-a_{i-1}Ap_{i-1}$ we have
\begin{eqnarray*}
     p_i &= &\gamma_{i-1} \left( \frac{r_{i-1}-r_i}{a_{i-1}} \right) - \sigma_{i-1}p_{i-1}, \qquad\qquad\qquad\quad \ \!  i = 1, \\
     p_i &= &\gamma_{i-1} \left( \frac{r_{i-1}-r_i}{a_{i-1}} \right) - \sigma_{i-1}p_{i-1} - \omega_{i-1}p_{i-2}, \qquad i \geq 2,
\end{eqnarray*}
so that from Theorem \ref{teo:properties}
    $$ r_i^Tp_i = - \frac{\gamma_{i-1}}{a_{i-1}}\|r_i\|^2. $$
The latter relation and (\ref{eq:reduction_error}) yield (\ref{equ:quad}).

As regards (\ref{equ:quad_2}), since $Ay_{m+1}=b$ then $b \in R(A)$, and from Table \ref{tab_CG-2step} then $r_i \in {\cal K}_i(b,A) \subseteq R(A)$, $i=0,\dots,m$, where ${\cal K}_{i+1}(b,A) \supseteq{\cal K}_i(b,A)$. In addition, by the definition of Moore-Penrose pseudoinverse matrix (see \cite{CM79}), and since $y_{m+1}$ is a solution of (\ref{system}) we have
\begin{eqnarray}
    Pr_{R(A)}(y_{m+1}) & = & A^+b \ = \ A^+(r_0+Ay_0)   \nonumber \\
                       &   &                            \nonumber \\
                       & = & A^+r_0 + Pr_{R(A)}(y_0).   \label{equ:proj}
\end{eqnarray}
Moreover, $y_{m+1}=y_0+\sum_{i=0}^m a_ip_i$ and by induction $p_i \in {\cal K}_i(b,A) \subseteq R(A)$, thus
\begin{eqnarray}
    Pr_{R(A)}(y_{m+1}) & = & Pr_{R(A)}(y_0) + Pr_{R(A)}\left(\sum_{i=0}^ma_ip_i\right)   \nonumber \\
                       & = & Pr_{R(A)}(y_0) + \sum_{i=0}^ma_ip_i.   \label{equ:proj_2}
\end{eqnarray}
By (\ref{equ:proj}), (\ref{equ:proj_2}) and recalling that for $CD$ we have $p_i^Tr_i= p_i^T(r_{i-1}-a_{i-1}Ap_{i-1}) =p_i^Tr_{i-1}= \cdots = p_i^Tr_0$, we obtain
    $$ A^+r_0 = \sum_{i=0}^m a_ip_i = \sum_{i=0}^m \frac{p_i^Tr_i}{p_i^TAp_i}p_i = \sum_{i=0}^m \frac{p_ip_i^T}{p_i^TAp_i}r_0, $$
which yields (\ref{equ:quad_2}).
$\hfill \Box$
\end{proof}

Observe that the result in (\ref{equ:optimal_2dim_mani}) may be seen as a consequence of the Theorem 3.6 in \cite{H80}, which holds for a general quadratic functional $g(x)$.
\begin{corollary}{\em \bf [Inverse Approximation]}
Let Assumption~\ref{assu_0} hold and suppose that $Ay_{m+1}=b$, where $y_{m+1}$ is computed by $CD$ and $m=n-1$. Then, we have
    $$ A^{-1} = \sum_{i=0}^{n-1} \frac{p_ip_i^T}{p_i^TAp_i}. $$

\end{corollary}
\begin{proof}
The proof follows from (\ref{equ:quad_2}), recalling that the directions $p_0, \dots,p_{n-1}$ are linearly independent and when $A$ is nonsingular $A^{-1} \equiv A^+$.
$\hfill \Box$
\end{proof}

\section{Basic Relation Between the CG and $CD$}
\label{sec:relaz_CG-CG_2step}
Observe that the geometry of vectors $\{p_k\}$ and $\{r_k\}$ in
$CD$ might be substantially different with respect to the CG.
Indeed, in the latter scheme the relation
$p_k=r_k+\beta_{k-1}p_{k-1}$ implies $r_k^Tp_k = \|r_k\|^2
> 0$, for any $k$. On the contrary, for the $CD$, using relation $r_k=r_{k-1}-a_{k-1}Ap_{k-1}$
and Theorem \ref{teo:properties} we have that possibly $r_k^Tp_k \neq \|r_k\|^2$ and
\begin{eqnarray*}
     \frac{p_k^TAp_k}{p_{k-1}^TAp_{k-1}} & = & \gamma_{k-1}\frac{(Ap_{k-1})^TAp_k}{p_{k-1}^TAp_{k-1}} \ = \ - \frac{\gamma_{k-1}\|r_k\|^2}{a_k a_{k-1}p_{k-1}^TAp_{k-1}} \\
                & = & - \gamma_{k-1}\frac{\|r_k\|^2 p_k^TAp_k}{(r_k^Tp_k) (r_{k-1}^Tp_{k-1})},
\end{eqnarray*}
so that when $A \succ 0$ we obtain
\begin{equation}
     \gamma_{k-1}(r_k^Tp_k)(r_{k-1}^Tp_{k-1}) < 0.  \label{eq:sign}
\end{equation}
The latter result is a consequence of the fact that in the $CD$ class, the direction $p_k$ is not generated
directly using the vector $r_k$. In addition, a similar conclusion also holds if we compute the quantity  $p_k^Tp_j > 0$, $k \not = j$, for both the CG and the $CD$ (see also Theorem 5.3 in \cite{HS52}).

As another difference between the CG and $CD$, we have that in the first algorithm the coefficient $\beta_{k-1}$, at Step $k$ in Table~\ref{tab_CG}, is always positive. On the other hand, the coefficients $\gamma_{k-1}$, $\sigma_{k-1}$ and $\omega_{k-1}$ at Step $k$ of Table~\ref{tab_CG-2step} might be possibly negative.

We also observe that the CG in Table~\ref{tab_CG} simply
stores at Step $k$ the vectors $r_{k-1}$ and $p_{k-1}$, in order
to compute respectively $r_{k}$ and $p_{k}$. On the other hand, at
Step $k$ the $CD$ requires the storage of one additional vector,
which contains some information from iteration $k-2$. The idea of
storing at Step $k$ some information from iterations preceding
Step $k-1$ is not new for Krylov-based methods. Some examples,
which differ from our approach, may be found in \cite{G97}, for
unsymmetric linear systems.

In any case, it is not difficult to verify that the CG may be equivalently obtained from $CD$, setting $\gamma_{k-1}=-\alpha_{k-1}$, for $k=1,2,\ldots$, in Table \ref{tab_CG-2step}. Indeed, though in Table \ref{tab_CG} the coefficient $\beta_{k-1}$ {\em explicitly} imposes the conjugacy only between $p_k$ and $p_{k-1}$, the pair $(\alpha_{k-1},\beta_{k-1})$ implicitly imposes both the conditions (\ref{two_star}) for the CG. Now, by (\ref{eq:altern_CG}) and comparing with Step $k$ of Table \ref{tab_CG-2step}, we want to show that setting $\gamma_{k-1}=-\alpha_{k-1}$ in Table \ref{tab_CG-2step} we obtain
\begin{equation}
     \left\{\begin{array}{lcl}
            \sigma_{k-1} = -(1+ \beta_{k-1}), & \quad & k \geq 1,    \\
                    & &      \\
            \omega_{k-1} = \beta_{k-2}, & \quad & k \geq 2,
       \end{array}   \right.       \label{equ:equiv_coeff}
\end{equation}
which implies that $CD$ reduces equivalently to the CG.
\\
For the CG $r_i^Tr_j=0$, for $i \not = j$, and $p_i^Tr_i = \|r_i\|^2$,  so that
    $$\beta_{k-1}:= \frac{\|r_k\|^2}{\|r_{k-1}\|^2} = - \frac{r_k^T(\alpha_{k-1}Ap_{k-1})}{\|r_{k-1}\|^2} = - \frac{r_k^TAp_{k-1}}{p_{k-1}^TAp_{k-1}}.$$
Thus, recalling that $r_{k-1} = r_{k-2}-\alpha_{k-2}Ap_{k-2}$ and $p_{k-1} = r_{k-1}+\beta_{k-2}p_{k-2}$,  we obtain for $\gamma_{k-1}=-\alpha_{k-1}$, with $k \geq 2$,
\begin{eqnarray}
    -(1+ \beta_{k-1}) & = & - \frac{p_{k-1}^TAp_{k-1} - r_k^TAp_{k-1}}{p_{k-1}^TAp_{k-1}}    \nonumber \\
                      & = & - \frac{(p_{k-1}- r_{k-1}+ \alpha_{k-1}Ap_{k-1})^TAp_{k-1}}{p_{k-1}^TAp_{k-1}}  \nonumber \\
                      & = & - \alpha_{k-1} \frac{\|Ap_{k-1}\|^2}{p_{k-1}^TAp_{k-1}} \ = \ \sigma_{k-1}  \label{equ:sigma}
\end{eqnarray}
and
\begin{eqnarray}
    \beta_{k-2}       & = & - \frac{r_{k-1}^TAp_{k-2}}{p_{k-2}^TAp_{k-2}} \ = \
                            \frac{\|r_{k-1}\|^2}{\alpha_{k-2}} \frac{1}{p_{k-2}^TAp_{k-2}}  \nonumber  \\
                      & = & \frac{\alpha_{k-1}}{\alpha_{k-2}}  \frac{p_{k-1}^TAp_{k-1}}{p_{k-2}^TAp_{k-2}} \ = \
                             \omega_{k-1}.  \label{equ:omega}
\end{eqnarray}
Finally, 
it is worth noticing that for $CD$ the following two properties hold, for any $k \geq 2$ ((i)-(ii) also hold for $k=1$, with obvious modifications to (i)):
\begin{itemize}
    \item[(i)] $\quad \displaystyle r_k^Tp_k = r_k^T \left[ \gamma_{k-1} \left( \frac{r_{k-1}-r_k}{a_{k-1}} \right) - \sigma_{k-1}p_{k-1}
                - \omega_{k-1}p_{k-2} \right] = - \frac{\gamma_{k-1}}{a_{k-1}}\|r_k\|^2$
    \item[(ii)] $\quad \displaystyle r_k^TAp_k = r_k^T \left( \frac{r_k-r_{k+1}}{a_k} \right)  = \frac{1}{a_k}\|r_k\|^2 = \frac{\|r_k\|^2}{r_k^Tp_k}p_k^TAp_k$,
\end{itemize}
which indicate explicitly a difference with respect to the CG. Indeed, for any $\gamma_{k-1} \not = -a_{k-1}$  we have respectively from (i) and (ii)
    $$ \begin{array}{c}
            r_k^Tp_k \not = \|r_k\|^2   \\
                \   \\
            r_k^TAp_k \not = p_k^TAp_k.
                   \end{array} $$
\begin{figure*}
\begin{center}
\includegraphics[width=0.76\textwidth]{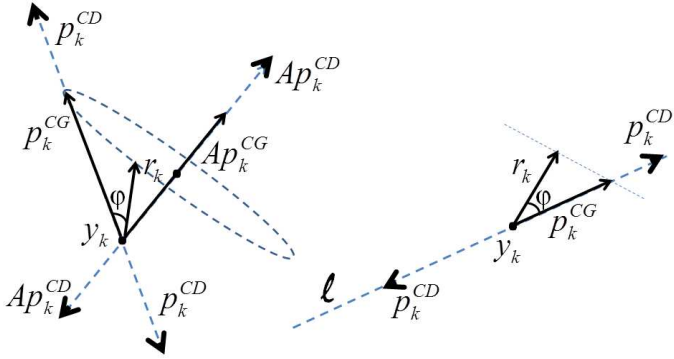}
\end{center}
\caption{At the $k$th iteration of the CG and $CD$, the directions $p_k^{CG}$ and $p_k^{CD}$ are respectively generated, along the line $\ell$. Applying the CG, the vectors $p_k^{CG}$ and $r_k$ have the same orthogonal projection on $Ap_k^{CG}$, since $(p_k^{CG})^TAp_k^{CG} = r_k^TAp_k^{CG}$. Applying $CD$, the latter equality with  $p_k^{CD}$ in place of $p_k^{CG}$ is not necessarily satisfied}
\label{fig:AAA_b}       
\end{figure*}
\noindent
Figure \ref{fig:AAA_b} clarifies the geometry of items (i) and (ii) for both the CG and $CD$.
\\
Relations (\ref{equ:sigma})-(\ref{equ:omega}) suggest that the sequence $\{\gamma_k\}$ must satisfy specific conditions in order to reduce $CD$ equivalently to the CG. For a possible generalization of the latter conclusion, consider that equalities (\ref{equ:equiv_coeff}) are by (\ref{eq:altern_CG}) sufficient conditions in order to reduce $CD$ equivalently to the CG. Thus, now we want to study general conditions on the sequence $\{\gamma_k\}$, such that (\ref{equ:equiv_coeff}) are satisfied. By (\ref{equ:equiv_coeff}) we have
    \[ -(1+\omega_k) = \sigma_{k-1}, \]
which is equivalent from Table \ref{tab_CG-2step} to
\begin{equation}    \label{equ_gamma_1}
     - \left(\gamma_{k-1}\|Ap_{k-1}\|^2 + p_{k-1}^TAp_{k-1} \right) = \frac{\gamma_k}{\gamma_{k-1}}p_k^TAp_k
\end{equation}
or
\begin{equation}    \label{equ_gamma_2}
     - \gamma_{k-1}^2\|Ap_{k-1}\|^2 - \gamma_{k-1} p_{k-1}^TAp_{k-1} -\gamma_kp_k^TAp_k =0.
\end{equation}
The latter equality, for $k \geq 1$, and the choice of $\sigma_0$ in Table \ref{tab_CG-2step} yield the following conclusions.
\begin{table}
\caption{The new $CD$-red class for solving (\ref{system}), obtained by setting at Step $k$ of $CD$ the parameter $\gamma_k$ as in relation (\ref{equ:relaz_gamma}).}
\begin{center}
\fbox{
\begin{tabular}{l}
    \   \\
\multicolumn{1}{c}{\bf The $CD$-red class}  \\
    \   \\
{\bf Step $0$:} $\quad\!$ Set $k=0$, $y_0 \in \mathbb{R}^n$, $r_0:=b-Ay_0$. \\
\hspace*{1cm} $\quad$ If $r_0 = 0$, then STOP. Else, set $p_0 := r_0$, $k=k+1$.  \\
\hspace*{1cm} $\quad$ Compute $a_0:=r_0^Tp_0 / p_0^TAp_0$, \ $\gamma_0 := -a_0$, \\
\hspace*{1cm} $\quad$ $y_1 := y_0 + a_0 p_0$, \ $r_1 := r_0 - a_0 Ap_0$.  \\
\hspace*{1cm} $\quad$ If $r_1 = 0$, then STOP. Else, set $\sigma_0:= \gamma_0\|Ap_0\|^2 / p_0^TAp_0$, \ $\beta_0 = -(1+\sigma_0) $ \\
\hspace*{1cm} $\quad$ $p_1 := r_1 + \beta_0p_0$, \  $k = k+1$. \\
{\bf Step $k$:} $\quad\!$ Compute $a_{k-1}:= r_{k-1}^Tp_{k-1} / p_{k-1}^TAp_{k-1}$, \\
\hspace*{1cm} $\quad$ $y_k := y_{k-1} + a_{k-1} p_{k-1}$, \ $r_k := r_{k-1} - a_{k-1}Ap_{k-1}$. \\
\hspace*{1cm} $\quad$ If $r_k = 0$, then STOP. Else, use (\ref{equ:relaz_gamma}) to compute $\gamma_{k-1}$.    \\
\hspace*{1cm} $\quad$ Set $\sigma_{k-1}:= \gamma_{k-1}\frac{\|Ap_{k-1}\|^2}{p_{k-1}^TAp_{k-1}}$, \ $\beta_{k-1} := -(1+\sigma_{k-1})$  \\
\hspace*{1cm} $\quad$ $p_k := r_k + \beta_{k-1}p_{k-1}$, \ $k = k+1$. \\
\hspace*{1cm} $\quad$ Go to {\bf Step $k$}.
\end{tabular}
}
\end{center}
\label{tab_CG-red}
\end{table}
\begin{lemma}{\em \bf [Reduction of $CD$]}
\label{prop:gamma_relation}
The scheme $CD$ in Table \ref{tab_CG-2step} can be rewritten as in Table \ref{tab_CG-red} (i.e. with the CG-like structure of Table \ref{tab_CG}), provided that the sequence $\{\gamma_k\}$ satisfies $\gamma_0 := -a_0$ and
\begin{equation}
     \gamma_k := - \frac{\gamma_{k-1}^2\|Ap_{k-1}\|^2 + \gamma_{k-1} p_{k-1}^TAp_{k-1}}{p_k^TAp_k},    \qquad k \geq 1.  \label{equ:relaz_gamma}
\end{equation}
In particular, the positions $\gamma_i=-a_i$, $i \geq 0$, in $CD$ satisfy (\ref{equ:relaz_gamma}).
\end{lemma}
\begin{proof}
By the considerations which led to (\ref{equ_gamma_1})-(\ref{equ_gamma_2}), relation (\ref{equ:relaz_gamma}) yields (\ref{equ:equiv_coeff}), so that the scheme $CD$-red in Table \ref{tab_CG-red} follows from $CD$ with the position (\ref{equ:relaz_gamma}), and setting $\gamma_0=-a_0$.
\\
Furthermore, replacing in (\ref{equ:relaz_gamma}) the conditions $\gamma_i=-a_i$, $i \geq 1$, and recalling (i)-(ii), we obtain the condition $a_{k-1}^2\|Ap_{k-1}\|^2=\|r_{k-1}\|^2 + \|r_k\|^2$, which is immediately fulfilled using condition $r_k=r_{k-1}-a_{k-1}Ap_{k-1}$.
$\hfill \Box$
\end{proof}

Note that the $CD$-red scheme substantially is more similar to the CG than to $CD$. Indeed the conditions (\ref{two_star}), explicitly imposed at Step $k$ of $CD$, reduce to the unique condition $p_k^TAp_{k-1}=0$ in $CD$-red.
\\
The following result is a trivial consequence of Lemma \ref{lem:finite_convergence}, where the alternate use of CG and $CD$ steps is analyzed.
\begin{lemma}{\em \bf [Combined Finite Convergence]}
\label{lem:combined}
Let Assumption~\ref{assu_0} hold. Let $y_1, \ldots,y_h$ be the iterates generated by $CD$, with $h \leq n$ and $Ay_h=b$. Then, finite convergence is preserved (i.e. $Ay_h=b$) if the Step $\hat k$ of $CD$, with $\hat k \in \{k_1, \ldots, k_h\} \subseteq \{1, \ldots,h\}$, is replaced by the Step $\hat k$ of the CG.
\end{lemma}
\begin{proof}
First observe that both in Table \ref{tab_CG} and Table \ref{tab_CG-2step}, for any $k \leq h$, the quantity $\|r_k\|>0$ is computed. Thus, in Table \ref{tab_CG} the coefficient $\beta_{k-1}$ is well defined for any $n > k \geq 1$. Now, by Table \ref{tab_CG-2step}, setting at Step $\hat k \in \{k_1, \ldots, k_h\} \subseteq \{1, \ldots,h\}$ the following
    $$ \left\{\begin{array}{lcl}
            \gamma_{\hat k-1} = -a_{\hat k -1} & \quad & {\rm if \ } \hat k \geq 1 \\
                    \       \\
            \sigma_{\hat k-1} = -(1+\beta_{\hat k -1}) & \quad & {\rm if \ } \hat k \geq 1 \\
                    \       \\
            \omega_{\hat k-1} = \beta_{\hat k -2} & \quad & {\rm if \ } \hat k \geq 2, \\
       \end{array} \right.$$
the Step $\hat k$ of $CD$ coincides formally with the Step $\hat k$ of CG. Thus, finite convergence with $Ay_h=b$ is proved recalling that Lemma \ref{lem:finite_convergence} holds for any choice of the sequence $\{\gamma_k\}$, with $\gamma_k \not = 0$.
$\hfill \Box$
\end{proof}

\section{Relation Between the Scaled-CG and $CD$}
\label{sec:scaled}
Similarly to the previous section, here we aim at determining the relation between our proposal in Table \ref{tab_CG-2step} and the scheme of the {\em scaled}-CG in Table \ref{tab_scaled-CG} (see also \cite{H80}, page 125).
\begin{table}
\caption{The scaled-CG algorithm for solving (\ref{system}).}
\begin{center}
\fbox{
\begin{tabular}{l}
    \   \\
\multicolumn{1}{c}{\bf The Scaled-CG method}  \\
    \   \\
{\bf Step $0$:} $\quad\!$ Set $k=0$, \ $y_0 \in \mathbb{R}$, \ $r_{0}:= b-A y_0$.  \\
\hspace*{1cm} $\quad$ If $r_{0}=0$, then STOP. Else, set $p_0 := \rho_0 r_0$, \ $\rho_0>0$, \ $k=k+1$.   \\
{\bf Step $k$:} $\quad\!$ Compute $\alpha_{k-1}:= \rho_{k-1}\|r_{k-1}\|^2/p_{k-1}^{T}Ap_{k-1}$, \ $\rho_{k-1}>0$,  \\
\hspace*{1.0cm} $\quad$ $y_{k}:= y_{k-1} + \alpha_{k-1}p_{k-1}$, \ $r_{k}:= r_{k-1} - \alpha_{k-1}Ap_{k-1}$. \\ \hspace*{1.0cm} $\quad$ If $r_{k}=0$, then STOP. Else, set $\beta_{k-1}:= - p_{k-1}^TAr_{k}/p_{k-1}^TAp_{k-1}$ or \\ \hspace*{5.25cm} $\quad$$\beta_{k-1}:= \|r_{k}\|^2/(\rho_{k-1}\|r_{k-1}\|^2)$    \\
\hspace*{1.0cm} $\quad$ $p_{k}:= \rho_k(r_{k} + \beta_{k-1}p_{k-1})$, \ $\rho_{k}>0$, \ $k = k+1$, \\
\hspace*{1.0cm} $\quad$ Go to {\bf Step $k$}.
\end{tabular}
}
\end{center}
\label{tab_scaled-CG}
\end{table}
In \cite{H80} a motivated choice for the coefficients $\{\rho_k\}$ in the scaled-CG is also given. Here, following the guidelines of the previous section, we first rewrite the relation
    $$ p_{k+1} := \rho_{k+1}(r_{k+1}+\beta_k p_k), $$
at Step $k+1$ of the scaled-CG, as follows
\begin{eqnarray}
    p_{k+1} & = & \rho_{k+1}(r_k - \alpha_k A p_k) + \rho_{k+1}\beta_k p_k  \nonumber   \\
            & = & \rho_{k+1}\left[ \frac{p_k}{\rho_k} - \beta_{k-1} p_{k-1} - \alpha_k A p_k \right] + \rho_{k+1}\beta_k p_k  \nonumber   \\
            & = & -\rho_{k+1}\alpha_k A p_k + \rho_{k+1} \left( \beta_k + \frac{1}{\rho_k}\right) p_k - \rho_{k+1} \beta_{k-1} p_{k-1}.    \label{equ:p_equiv}
\end{eqnarray}
We want to show that for a suitable choice of the parameters $\{\gamma_k\}$, the $CD$ yields the recursion (\ref{equ:p_equiv}) of the scaled-CG, i.e. for a proper choice of $\{\gamma_k\}$ we obtain from CD a scheme equivalent to the scaled-CG. On this purpose let us set in $CD$
\begin{equation}
    \gamma_k = - \rho_{k+1} \alpha_k, \qquad k \geq 0,  \label{equ:choice_gamma}
\end{equation}
where $\alpha_k$ is given at Step $k$ of Table \ref{tab_scaled-CG}. Thus, by Table \ref{tab_CG-2step}
\begin{equation}
    \sigma_k  \ = \ \gamma_k \frac{\|Ap_k\|^2}{p_k^TAp_k} \ = \ - \rho_{k+1} \alpha_k \frac{\|Ap_k\|^2}{p_k^TAp_k}, \qquad k \geq 0, \label{equ:choice_sigma}
\end{equation}
and for $k \geq 1$
\begin{eqnarray}
    \omega_k & = & \frac{\gamma_k}{\gamma_{k-1}} \frac{p_k^TAp_k}{p_{k-1}^TAp_{k-1}} \ = \ \frac{\rho_{k+1} \alpha_k}{\rho_{k} \alpha_{k-1}} \frac{p_k^TAp_k}{p_{k-1}^TAp_{k-1}}.  \label{equ:choice_omega}
\end{eqnarray}
Now, comparing the coefficients in (\ref{equ:p_equiv}) with (\ref{equ:choice_gamma}), (\ref{equ:choice_sigma}) and (\ref{equ:choice_omega}), we want to prove that the choice (\ref{equ:choice_gamma}) implies
\begin{eqnarray}
    \sigma_k &=&-\rho_{k+1} \left( \beta_k + \frac{1}{\rho_k}\right),  \qquad  k \geq 0,  \label{equ_a}  \\
    \omega_k &=&\rho_{k+1} \beta_{k-1},  \qquad\qquad\qquad  k \geq 1,       \label{equ_b}
\end{eqnarray}
so that the $CD$ class yields equivalently the scaled-CG.

As regards (\ref{equ_a}), from Table \ref{tab_scaled-CG} we have for $k \geq 0$
\begin{eqnarray*}
    \beta_k + \frac{1}{\rho_k} & = & \frac{\frac{1}{\rho_k} p_k^TAp_k - r_{k+1}^TAp_k}{p_k^TAp_k} \ = \
                  \frac{\left(\frac{1}{\rho_k} p_k - r_{k+1} \right)^TAp_k}{p_k^TAp_k} \\
                  & = &  \frac{\left(\frac{1}{\rho_k} p_k -r_k + \alpha_k Ap_k \right)^TAp_k}{p_k^TAp_k} \\
                               & = & \frac{\left(r_k + \beta_{k-1}p_{k-1}-r_k + \alpha_k Ap_k \right)^TAp_k}{p_k^TAp_k} \ = \ \alpha_k \frac{\|Ap_k\|^2}{p_k^TAp_k},
\end{eqnarray*}
so that from (\ref{equ:choice_sigma}) the condition (\ref{equ_a}) holds, for any $k \geq 0$. As regards (\ref{equ_b})
from Step $k$ of Table \ref{tab_scaled-CG} we know that $\beta_{k-1}= \|r_k\|^2/(\rho_{k-1}\|r_{k-1}\|^2)$ and, since $r_k^Tp_{k-1}=0$, we obtain $r_k^Tp_k = \rho_k \|r_k\|^2$; thus, relation (\ref{equ:choice_gamma}) yields
    $$ \beta_{k-1} = \frac{\|r_k\|^2}{\rho_{k-1}\|r_{k-1}\|^2}  = \frac{\alpha_k}{\rho_k\alpha_{k-1}} \frac{p_k^TAp_k}{p_{k-1}^TAp_{k-1}} = \frac{\gamma_k}{\rho_{k+1}\gamma_{k-1}} \frac{p_k^TAp_k}{p_{k-1}^TAp_{k-1}}, \ \ k \geq 1.$$
Relation (\ref{equ_b}) is proved using the latter equality and (\ref{equ:choice_omega}).

\section{Matrix Factorization Induced by $CD$}
\label{sec:matr_fact}
We first recall that considering the CG in Table \ref{tab_CG} and setting at Step $h$
   $$  \begin{array}{l}
            P_h := \displaystyle \left(\frac{p_0}{\|r_0\|} \ \cdots \ \frac{p_h}{\|r_h\|}\right)    \\
                \   \\
            R_h := \displaystyle \left(\frac{r_0}{\|r_0\|} \ \cdots \ \frac{r_h}{\|r_h\|}\right),
        \end{array}$$
along with
    $$                      L_h := \left(\begin{array}{ccccc}
                                            1 & &  &  &     \\
                                              & &  &  &     \\
                                            -\sqrt{\beta_0}  & 1 &  &  &     \\
                                             & &  &  &     \\
                                              & -\sqrt{\beta_1}  &  1 &  &       \\
                                               & &  &  &     \\
                                              &   & \ddots   &  1 &             \\
                                               & &  &  &     \\
                                              &   &    &  -\sqrt{\beta_{h-1}} \  & 1
                                        \end{array} \right) \in \mathbb{R}^{h \times h}  $$
and $D_h := {\rm diag}_i\{1/\alpha_i\}$, we obtain the three matrix relations
\begin{eqnarray}
    P_h L_h^T & = & R_h \label{equ:first_prop}  \\
    AP_h & = & R_h L_h D_h - \frac{\sqrt{\beta_h}}{\alpha_h} \frac{r_{h+1}}{\|r_{h+1}\|}e_h^T    \label{equ:second_prop}  \\
    R_h^TAR_h & = & T_h \ = \ L_h D_h L_h^T.    \label{equ:third_prop}
\end{eqnarray}
Then, in this section we are going to use the iteration in Table \ref{tab_CG-2step} in order to possibly recast relations  (\ref{equ:first_prop})-(\ref{equ:third_prop}) for $CD$.

On this purpose, from Table \ref{tab_CG-2step} we can easily draw the following relation between the sequences $\{p_0, p_1, \ldots\}$ and $\{r_0, r_1, \ldots\}$
    $$ \begin{array}{l}
            p_0 \ = \ r_0 \\
                \  \\
            \displaystyle p_1 \ = \ \frac{\gamma_0}{a_0}(r_0-r_1)-\sigma_0p_0  \\
                \  \\
            \displaystyle p_i \ = \ \frac{\gamma_{i-1}}{a_{i-1}}(r_{i-1}-r_i)-\sigma_{i-1}p_{i-1}-\omega_{i-1}p_{i-2}, \qquad\qquad i=2,3,\ldots,
       \end{array} $$
and introducing the positions
    $$  \begin{array}{l}
           \displaystyle  P_h := (p_0 \ p_1 \ \cdots \ p_h)    \\
                \   \\
           \displaystyle  R_h := (r_0 \ r_1 \ \cdots \ r_h)  \\
                \   \\
           \displaystyle  \bar R_h := \left( \frac{r_0}{\|r_0\|} \ \cdots \ \frac{r_h}{\|r_h\|} \right),
        \end{array}$$
along with the matrices
    $$                      U_{h,1} := \left(\begin{array}{ccccccc}
                                            1 & \sigma_0 & \omega_1 & 0 & \cdots & \cdots & 0   \\
                                              & 1 & \sigma_1 & \omega_2 & 0 & \cdots & 0   \\
                                              &   &  1 & \sigma_2 & \ddots & 0 & \vdots    \\
                                              &   &    &  1 & \ddots & \ddots & 0          \\
                                              &   &    &    & \ddots & \ddots & \omega_{h-1}          \\
                                              &   &    &    &        & \ddots & \sigma_{h-1}    \\
                                              &   &    &    &        &  & 1
                                        \end{array} \right) \in \mathbb{R}^{(h+1)\times(h+1)},  $$
    $$                      U_{h,2} := \left(\begin{array}{cccccc}
                                            \|r_0\| & \|r_0\| & 0 & \cdots & \cdots & 0             \\
                                                & & & & &  \\
                                              & -\|r_1\| & \|r_1\| & 0 & \cdots & 0            \\
                                               & & & & &  \\
                                              &   &   -\|r_2\| & \ \|r_2\| & 0  & \vdots               \\
                                               & & & & &  \\
                                              &   &    &    \ddots & \ddots &   0                   \\
                                               & & & & &  \\
                                              &   &    &    & -\|r_{h-1}\| & \ \|r_{h-1}\|                  \\
                                               & & & & &  \\
                                               &    &    &        &        & -\|r_h\|
                                        \end{array} \right) \in \mathbb{R}^{(h+1)\times(h+1)} $$
and
    $$           D_h := {\rm diag} \left\{ 1 \ , \ \  \displaystyle {\rm diag}_{\atop \hspace*{-.8cm}i=0, \ldots, h-1} \{\gamma_i/a_i\} \right\} \in \mathbb{R}^{(h+1)\times(h+1)},  $$
we obtain after $h-1$ iterations of $CD$
    $$ P_h U_{h,1} = \bar R_h U_{h,2} D_h, $$
so that
    $$ P_h = \bar R_h U_{h,2}D_h U_{h,1}^{-1} = \bar R_h U_h, $$
where $U_h = U_{h,2} D_h U_{h,1}^{-1}$. Now, observe that $U_h$ is upper triangular since $U_{h,2}$ is upper bidiagonal, $D_h$ is diagonal and $U_{h,1}^{-1}$ may be easily seen to be upper triangular. As a consequence, recalling that $p_0, \ldots, p_h$ are mutually conjugate we have
    $$ \bar R_h^TA \bar R_h = U_h^{-T} {\rm diag}_i \{p_i^TAp_i\} U_h^{-1}, $$
and in case $h=n-1$, again from the conjugacy of $p_0, \ldots, p_{n-1}$
    $$ P_{n-1}^TAP_{n-1} = U_{n-1}^T\bar R_{n-1}^T A \bar R_{n-1} U_{n-1} = {\rm diag}_{\atop \hspace*{-.8cm}i=0, \ldots, h-1} \{p_i^TAp_i\}.$$
From the orthogonality of $\bar R_{n-1}$, along with relation
    $$ \det(U_{n-1}) = \|r_0\| \prod_{j=1}^{n-1}\left( -\frac{\|r_j\| \gamma_{j-1}}{a_{j-1}} \right) =
                       \left(\prod_{i=0}^{n-1}\|r_i\| \right) \left( \prod_{i=0}^{n-2}-\frac{\gamma_{i}}{a_i} \right),  $$
 we have
    $$ \det\left(U_{n-1}^T\bar R_{n-1}^T A \bar R_{n-1} U_{n-1}\right) = \prod_{i=0}^{n-1}p_i^TAp_i \quad \Longleftrightarrow \quad \det(A)= \frac{\displaystyle \prod_{i=0}^{n-1}p_i^TAp_i}{\det(U_{n-1})^2}.$$
Thus, in the end
\begin{equation}
     \det(A)= \left[\prod_{i=0}^{n-1} \frac{p_i^TAp_i}{\|r_i\|^2} \right]\cdot \frac{\left[\displaystyle\prod_{i=0}^{n-2} a_i^2\right]}{\left[\displaystyle\prod_{i=0}^{n-2} \gamma_i^2\right]}. \label{eq:determ}
\end{equation}
Note that the following considerations hold:
\begin{itemize}
    \item for  $\gamma_i = \pm a_i$ (which includes the case $\gamma_i = - a_i$, when by Lemma \ref{prop:gamma_relation} $CD$ reduces equivalently to the CG), by (i) of Section \ref{sec:relaz_CG-CG_2step} $|p_k^Tr_k| = \|r_k\|^2$, so that we obtain the standard result (see also \cite{HS52})
    $$ \det(A) = \left[\prod_{i=0}^{n-1} \frac{p_i^TAp_i}{\|r_i\|^2} \right] = \prod_{i=0}^{n-1} \frac{1}{ a_i}; $$
    \item if in general $|\gamma_i| \not = | a_i|$ we obtain the general formula (\ref{eq:determ}).
\end{itemize}

\section{Issues on the Conjugacy Loss for $CD$}
\label{sec:conjugacy_loss}
Here we consider a simplified approach to describe the conjugacy
loss for both the CG and $CD$, under
Assumption \ref{assu_0} (see also \cite{HS52} for a similar approach). Suppose that both the CG and
$CD$ perform Step $k+1$, and for numerical reasons a nonzero  {\em conjugacy error}
$\varepsilon_{k,j}$ respectively occurs between directions $p_k$
and $p_j$, i.e.
    $$ \varepsilon_{k,j} \ : = \ p_k^TAp_j \ \not = \ 0, \qquad j \leq k-1. $$
Then, we calculate the conjugacy error
\begin{eqnarray*}
    \varepsilon_{k+1,j} \ = \ p_{k+1}^TAp_j, \qquad\qquad j \leq k,
\end{eqnarray*}
for both the CG and $CD$. First observe that at
Step $k+1$ of Table \ref{tab_CG} we have
\begin{eqnarray}
   \varepsilon_{k+1,j} & = & \left( r_{k+1}+\beta_k p_k \right)^TAp_j   \label{equ:11}  \\
                       &   &      \nonumber    \\
                       & = & \left( p_k -\beta_{k-1}p_{k-1} - \alpha_k Ap_k \right)^TAp_j + \beta_k \varepsilon_{k,j}  \label{equ:12}  \\
                       &   &      \nonumber    \\
                       & = & (1+\beta_k) \varepsilon_{k,j} - \beta_{k-1} \varepsilon_{k-1,j} - \alpha_k
                       (Ap_k)^TAp_j.    \label{relation_3}
\end{eqnarray}
Then, from relation $Ap_j = (r_j - r_{j+1})/ \alpha_j$ and
relations (\ref{rela_1a})-(\ref{rela_2a}) we have for the CG
   $$ \displaystyle (Ap_k)^TAp_j \ = \ \left\{ \begin{array}{cl}
          \displaystyle - \frac{p_k^TAp_k}{\alpha_{k-1}},& \ \ \ \ \ \ \ j = k-1, \\
                                           &                        \\
              \emptyset,                   & \ \ \ \ \ \ \ j \leq
              k-2.
                                \end{array}  \right.  $$
Thus, observing that for the CG we have $\varepsilon_{i,i-1}=0$
and $\varepsilon_{i,i}=p_i^TAp_i$, $1 \leq i \leq k+1$, after some
computation we obtain from (\ref{rela_1a}), (\ref{rela_2a}) and (\ref{relation_3})
\begin{equation}
    \varepsilon_{k+1,j} \ = \ \left\{ \begin{array}{ll}
                     \emptyset,& \ \ \ \ \ \ \ j = k,  \\
                           &                   \\
                             \emptyset,& \ \ \ \ \ \ \ j = k-1,\\
                                       &                       \\
      (1+\beta_k) \varepsilon_{k,k-2}, & \ \ \ \ \ \ \ j = k-2,\\
                                       &                       \\
          (1+\beta_k) \varepsilon_{k,j}- \beta_{k-1} \varepsilon_{k-1,j} - \Sigma_{kj},  & \ \ \ \ \ \ \ j \leq
          k-3,
                                       \end{array}  \right.  \label{rela_3}
\end{equation}
where $\Sigma_{kj} \in \mathbb{R}$ summarizes the contribution of the
term $\alpha_k (Ap_k)^TAp_j$, due to a possible conjugacy loss.

Let us consider now for $CD$ a result similar to (\ref{rela_3}). We obtain the following relations for
$j \leq k$
\begin{eqnarray*}
   \varepsilon_{k+1,j} & = & p_{k+1}^TAp_j \ = \ \left( \gamma_k Ap_k -\sigma_k p_k  - \omega_k p_{k-1} \right)^TAp_j  \\
                       &       &            \\
                       &  =    & \gamma_k (Ap_k)^TAp_j - \sigma_k \varepsilon_{k,j} - \omega_k \varepsilon_{k-1,j}  \\
                       &       &            \\
                       &   =   & \frac{\gamma_k}{\gamma_j}(Ap_k)^T \left( p_{j+1} +\sigma_j p_j  + \omega_j p_{j-1} \right) - \sigma_k \varepsilon_{k,j} - \omega_k \varepsilon_{k-1,j}  \\
                       &       &            \\
                       &   =   & \frac{\gamma_k}{\gamma_j} \varepsilon_{k,j+1} + \left(\frac{\gamma_k}{\gamma_j} \sigma_j - \sigma_k \right) \varepsilon_{k,j}+ \frac{\gamma_k}{\gamma_j} \omega_j \varepsilon_{k,j-1} - \omega_k \varepsilon_{k-1,j},
\end{eqnarray*}
and considering now relations (\ref{two_star}), the conjugacy
among directions $p_0,p_1, \ldots, p_k$ satisfies
\begin{equation}
    \varepsilon_{h,l} \ = \ p_h^TAp_l \ = \ 0, \qquad \qquad {\rm for \ any} \ \mid h-l \mid \ \in \{1,2\}.  \label{three_star}
\end{equation}
Thus, relation (\ref{two_x}) and the expression of the
coefficients in $CD$ yields for $\varepsilon_{k+1,j}$ the expression
\begin{equation}
      \left\{ \begin{array}{ll}
       \emptyset,  & \ \ \ \ \ \ \ \ j = k,  \\
               &                    \\
         \emptyset,  & \ \ \ \ \ \ \ \ j = k-1,  \\
                    &                \\
         \displaystyle \frac{\gamma_k}{\gamma_{k-2}} \omega_{k-2} \varepsilon_{k,k-3},  & \ \ \ \ \ \ \ \ j = k-2, \\
                    &                \\
         \displaystyle \left(\frac{\gamma_k}{\gamma_{k-3}}\sigma_{k-3} - \sigma_k \right) \varepsilon_{k,k-3}+ \frac{\gamma_k}{\gamma_{k-3}}\omega_{k-3} \varepsilon_{k,k-4},  & \ \ \ \ \ \ \ \ j = k-3,  \\
                    &                \\
        \displaystyle \frac{\gamma_k}{\gamma_{j}}\varepsilon_{k,j+1} + \left(\frac{\gamma_k}{\gamma_{j}}\sigma_{j} - \sigma_k \right) \varepsilon_{k,j}+ \frac{\gamma_k}{\gamma_{j}} \omega_j \varepsilon_{k,j-1} - \omega_k\varepsilon_{k-1,j},  & \ \ \ \ \ \ \ \ j \leq k-4.
                                   \end{array} \right.
    \label{rela_4}
\end{equation}
Finally, comparing relations (\ref{rela_3}) and (\ref{rela_4}) we
have
\begin{itemize}
   \item in case $j=k-2$ the conjugacy error $\varepsilon_{k+1,k-2}$ is nonzero for both the CG and $CD$, as expected. However, for the CG
    $$|\varepsilon_{k+1,k-2}| \ > \ |\varepsilon_{k,k-2}|$$
since $(1+\beta_k) > 1$, which theoretically can lead to an
harmful amplification of conjugacy errors. On the contrary, for
 $CD$ the positive quantity $|\gamma_k \omega_{k-2}/ \gamma_{k-2}|$ in the
expression of $\varepsilon_{k+1,k-2}$ can be possibly smaller than
one.

    \item choosing the sequence $\{\gamma_k\}$ such that
\begin{equation}
    \left|\frac{\gamma_k}{\gamma_{k-i}}\right| \ll 1  \qquad {\rm and/or} \qquad
    \left|\frac{\gamma_k}{\gamma_{k-i}}\omega_{k-i}\right| \ll 1, \qquad i=2,3,\ldots    \label{eq:seq_gamma}
\end{equation}
from (\ref{rela_4}) the effects of conjugacy loss may be attenuated. Thus, a strategy to update the sequence $\{\gamma_k\}$ so that (\ref{eq:seq_gamma}) holds might be investigated.

\end{itemize}

\subsection{Bounds for the Coefficients of $CD$}
We want
to describe here the sensitivity of the coefficients $\sigma_k$
and $\omega_k$, at Step $k+1$ of $CD$, to
the condition number $\kappa (A)$. In particular, we want to
provide a comparison with the CG, in order to identify
possible advantages/disadvantages of our proposal. From Table
\ref{tab_CG-2step} and Assumption \ref{assu_0} we have
    $$
      |\omega_k|  \ = \ \left|\frac{\gamma_k}{\gamma_{k-1}}\frac{p_k^TAp_k}{p_{k-1}^TAp_{k-1}}\right|, \qquad
      |\sigma_k|  \ = \ \left|\gamma_k \frac{\|Ap_k\|^2}{p_k^TAp_k}\right|,
    $$
so that
\begin{eqnarray}
     \left\{ \begin{array}{l}
            |\omega_k| \geq \displaystyle \left|\frac{\gamma_k}{\gamma_{k-1}}\right| \frac{\lambda_m(A)\|p_k\|^2}{\lambda_M(A)\|p_{k-1}\|^2} = \left|\frac{\gamma_k}{\gamma_{k-1}}\right| \frac{1}{\kappa(A)} \frac{\|p_k\|^2}{\|p_{k-1}\|^2}  \\
        \  \\
            |\omega_k| \leq \displaystyle \left|\frac{\gamma_k}{\gamma_{k-1}}\right| \frac{\lambda_M(A)\|p_k\|^2}{\lambda_m(A)\|p_{k-1}\|^2} = \left|\frac{\gamma_k}{\gamma_{k-1}} \right| \kappa (A) \frac{\|p_k\|^2}{\|p_{k-1}\|^2},
                                                                     \end{array}  \right.  \label{sigma_{k-2}}
\end{eqnarray}
and
\begin{eqnarray}
    \left\{ \begin{array}{l}
            |\sigma_k| \geq \displaystyle |\gamma_k| \frac{\lambda^2_m(A)\|p_k\|^2}{\lambda_M(A)\|p_k\|^2} = |\gamma_k| \frac{\lambda_m(A)}{\kappa(A)} \\
        \  \\
            |\sigma_k| \leq \displaystyle |\gamma_k| \frac{\lambda^2_M(A)\|p_k\|^2}{\lambda_m(A)\|p_k\|^2} = |\gamma_k| \lambda_M(A) \kappa (A).
                                                                     \end{array}  \right. \label{sigma_{k-1}}
\end{eqnarray}
On the other hand, from Table \ref{tab_CG} we obtain for the CG
    $$ \beta_k  \ = \ -\frac{r_{k+1}^TAp_k}{p_{k}^TAp_{k}} \ = \ -1 + \alpha_k \frac{\|Ap_k\|^2}{p_{k}^TAp_k} \ = \
                       \\ -1 + \frac{\|r_k\|^2}{p_{k}^TAp_k}
                       \frac{\|Ap_k\|^2}{p_{k}^TAp_k}, $$
so that, since $\beta_k > 0$ and using relation $\|r_k\| \leq
\|p_k\|$, along with $p_k^TAp_k = r_k^TAr_k - \frac{\|r_k\|^4}{\|r_{k-1}\|^4} p_{k-1}^TAp_{k-1}>0$, we have
\begin{equation}
\ \left\{
\begin{array}{l}
                                        \displaystyle \beta_k \geq \max \left\{0, -1 + \frac{\|r_k\|^2}{r_{k}^TAr_k} \frac{\lambda_m(A)}{\kappa (A)} \right\}
                                        \geq \max \left\{0, -1 + \frac{1}{[\kappa (A)]^2} \right\}
                                        = 0  \\
                                        \    \\
                                        \displaystyle \beta_k \leq -1 + \frac{\|p_k\|^2}{p_{k}^TAp_k} \lambda_M(A)\kappa (A)
                                        \leq -1 + [\kappa (A)]^2.
                       \end{array}   \right.   \label{new_bound_beta}
\end{equation}

In particular, this seems to indicate that on those problems where the quantity
$|\gamma_k| \lambda_M(A)$ is reasonably small, $CD$ might be competitive. However, as expected, high values for
$\kappa (A)$ may determine numerical instability for both the CG
and $CD$. In addition, observe that any
conclusion on the comparison between the numerical performance of the CG and $CD$,
depends both on the sequence $\{\gamma_k\}$ and on how tight are the bounds (\ref{sigma_{k-1}}) and
(\ref{new_bound_beta}) for the problem in hand.
\begin{table}
\caption{The $CD$ class for solving the linear system $\bar A \bar y = \bar b$ in (\ref{prec_system}).}
\begin{center}
\fbox{
\begin{tabular}{l}
    \   \\
\multicolumn{1}{c}{\bf The $CD$ class for (\ref{prec_system})}  \\
    \   \\
{\bf Step $0$:} $\quad\!$ Set $k=0$, $ \bar y_0 \in \mathbb{R}^n$, $\bar r_0:= \bar b- \bar A \bar y_0$, $\bar \gamma_0 \in \mathbb{R} \setminus \{0\}$. \\
\hspace*{1cm} $\quad$ If $ \bar r_0 = 0$, then STOP. Else, set $ \bar p_0 := \bar r_0$, \ $k = k+1$.  \\
\hspace*{1cm} $\quad$ Compute $ \bar a_0:=\bar r_0^T{\bar p_0} / \bar p_0^T \bar A \bar p_0$, \\
\hspace*{1cm} $\quad$ $ \bar y_1 :=  \bar y_0 +  \bar a_0  \bar p_0$, \ $ \bar r_1 :=  \bar r_0 -  \bar a_0  \bar A \bar p_0$.  \\
\hspace*{1cm} $\quad$ If $ \bar r_1 = 0$, then STOP. Else, set $ \bar \sigma_0:= \bar \gamma_0 \| \bar A \bar p_0\|^2 / \bar p_0^T \bar A \bar p_0$, \\
\hspace*{1cm} $\quad$ $ \bar p_1 :=  \bar \gamma_0 \bar A \bar p_0 - \bar \sigma_0 \bar p_0$, \  $k = k+1$. \\
{\bf Step $k$:} $\quad\!$ Compute $ \bar a_{k-1}:=  \bar r_{k-1}^T \bar p_{k-1} / \bar p_{k-1}^T \bar A \bar p_{k-1}$, $\bar \gamma_{k-1} \in \mathbb{R}\setminus \{0\}$, \\
\hspace*{1cm} $\quad$ $\bar y_k :=  \bar y_{k-1} +  \bar a_{k-1}  \bar p_{k-1}$, \ $ \bar r_k :=  \bar r_{k-1} -  \bar a_{k-1} \bar A \bar p_{k-1}$. \\
\hspace*{1cm} $\quad$ If $ \bar r_k = 0$, then STOP. Else, set \\
\hspace*{1cm} $\quad$ $ \bar \sigma_{k-1}:= \bar \gamma_{k-1} \frac{\| \bar A \bar p_{k-1}\|^2}{ \bar p_{k-1}^T \bar A \bar p_{k-1}}$, \ $ \bar \omega_{k-1} := \frac{\bar \gamma_{k-1}}{\bar \gamma_{k-2}} \frac{ \bar p_{k-1}^T \bar A \bar p_{k-1}}{ \bar p_{k-2}^T \bar A \bar p_{k-2}}$,\\
\hspace*{1cm} $\quad$ $ \bar p_k :=  \bar \gamma_{k-1} \bar A \bar p_{k-1} -  \bar \sigma_{k-1} \bar p_{k-1} -  \bar \omega_{k-1} \bar p_{k-2}$, \ $k = k+1$. \\
\hspace*{1cm} $\quad$ Go to {\bf Step $k$}.
\end{tabular}
}
\end{center}
\label{tab_prec_CG-2step}
\end{table}
\begin{table}
\caption{The preconditioned $CD$, namely $CD$$_{{\cal M}}$, for solving (\ref{system}).}
\begin{center}
\fbox{
\begin{tabular}{l}
    \   \\
\multicolumn{1}{c}{\bf The $CD$$_{{\cal M}}$ class}  \\
    \   \\
{\bf Step $0$:} $\quad\!$ Set $k=0$, $y_0 \in \mathbb{R}^n$, $r_0:=b-Ay_0$, $\bar \gamma_{0} \in \mathbb{R}\setminus \{0\}$, ${\cal M} \succ 0$. \\
\hspace*{1cm} $\quad$ If $r_0 = 0$, then STOP. Else, set $p_0 := {\cal M}r_0$, \ $k = k+1$.  \\
\hspace*{1cm} $\quad$ Compute $a_0:=r_0^Tp_0/p_0^TAp_0$, \\
\hspace*{1cm} $\quad$ $y_1 := y_0 + a_0 p_0$, \ $r_1 := r_0 - a_0 Ap_0$.  \\
\hspace*{1cm} $\quad$ If $r_1 = 0$, then STOP. Else, set $\sigma_0:= \bar \gamma_{0} \|Ap_0\|_{\cal M}^2/p_0^TAp_0$,\\
\hspace*{1cm} $\quad$ $p_1 := \bar \gamma_{0}{\cal M}(Ap_0) -\sigma_0p_0$, \  $k = k+1$. \\
{\bf Step $k$:} $\quad\!$ Compute $a_{k-1}:= r_{k-1}^Tp_{k-1}/p_{k-1}^TAp_{k-1}$, \ $\bar \gamma_{k-1} \in \mathbb{R} \setminus \{0\}$,  \\
\hspace*{1cm} $\quad$ $y_k := y_{k-1} + a_{k-1} p_{k-1}$, \ $r_k := r_{k-1} - a_{k-1}Ap_{k-1}$. \\
\hspace*{1cm} $\quad$ If $r_k = 0$, then STOP. Else, set \\
\hspace*{1cm} $\quad$ $\sigma_{k-1}:= \bar \gamma_{k-1} \frac{\|Ap_{k-1}\|_{\cal M}^2}{p_{k-1}^TAp_{k-1}}$, \ $ \omega_{k-1}:= \frac{\bar \gamma_{k-1}}{\bar \gamma_{k-2}}\frac{p_{k-1}^TAp_{k-1}}{p_{k-2}^TAp_{k-2}}$, \\\
\hspace*{1cm} $\quad$ $p_k := \bar \gamma_{k-1}{\cal M}(Ap_{k-1}) - \sigma_{k-1}p_{k-1} - \omega_{k-1}p_{k-2}$, \ $k = k+1$. \\
\hspace*{1cm} $\quad$ Go to {\bf Step $k$}.
\end{tabular}
}
\end{center}
\label{tab_prec_M_CG-2step}
\end{table}

\section{The Preconditioned $CD$ Class}
\label{sec:prec_CG_2step} In this section we introduce
preconditioning for the class $CD$, in order to
better cope with possible illconditioning of the matrix $A$ in
(\ref{system}).
\\
Let $M \in \mathbb{R}^{n \times n}$ be {\em nonsingular} and consider the linear system (\ref{system}). Since we have
\begin{eqnarray}
    Ay=b & \Longleftrightarrow & \left(M^T M \right)^{-1} A y =  \left(M^T M \right)^{-1} b  \label{4_stars} \\
         & \Longleftrightarrow & \left(M^{-T} A M^{-1} \right) M y = M^{-T} b  \nonumber \\
         & \Longleftrightarrow & \bar A \bar y = \bar b,   \label{prec_system}
\end{eqnarray}
where
\begin{equation}
        \bar A := M^{-T} A M^{-1}, \qquad \bar y := M y, \qquad \bar b := M^{-T} b,  \label{bar_quantities}
\end{equation}
solving (\ref{system}) is equivalent to solve (\ref{4_stars}) or (\ref{prec_system}). Moreover, any eigenvalue $\lambda_i$, $i=1,\ldots,n$, of $M^{-T}AM^{-1}$ is also an eigenvalue of $\left(M^TM\right)^{-1}A$. Indeed, if $(M^TM)^{-1}A z_i = \lambda_i z_i$, $i=1, \ldots,n$, then
    $$  \left(M^{-1}M^{-T} \right) A M^{-1} \left(M z_i\right) = \lambda_i z_i $$
so that
    $$ M^{-T}A M^{-1} \left(M z_i \right) = \lambda_i \left(Mz_i \right). $$
Now, let us motivate the importance of selecting a promising
matrix $M$ in (\ref{prec_system}), in order to reduce $\kappa
(\bar A)$ (or equivalently to reduce $\kappa [ (M^TM)^{-1} A ]$).
\\
Observe that under the Assumption~\ref{assu_0} and using standard
Chebyshev polynomials analysis, we can prove that in exact algebra
for both the CG and $CD$ the following relation
holds (see \cite{GV89} for details, and a similar analysis holds
for $CD$)
\begin{equation}
    \frac{\|y_k-y^\ast\|_A}{\|y_0-y^\ast\|_A} \leq 2 \left( \frac{\sqrt{\kappa (A)}-1}{\sqrt{\kappa (A)}+1} \right)^k,  \label{Chebyshev}
\end{equation}
where $Ay^\ast = b$. Relation
(\ref{Chebyshev}) reveals the strong dependency of the iterates
generated by the CG and $CD$, on $\kappa (A)$. In
addition, if the CG and $CD$ are used to solve
(\ref{prec_system}) in place of (\ref{system}), then the bound
(\ref{Chebyshev}) becomes
\begin{equation}
    \frac{\|y_k-y^\ast\|_A}{\|y_0-y^\ast\|_A} \leq 2 \left( \frac{\sqrt{\kappa [ (M^TM)^{-1} A ]}-1}{\sqrt{\kappa [ (M^TM )^{-1} A ] }+1} \right)^k,  \label{Chebyshev_prec}
\end{equation}
which definitely encourages to use the {\em
pre\-con\-di\-tio\-ner} $(M^TM)^{-1}$ whenever we have $\kappa [ (M^TM)^{-1} A
] < \kappa (A)$.
\\
On this guideline we want to introduce preconditioning in our
scheme $CD$, for solving the linear system
(\ref{prec_system}), where $M$ is non-singular. We do not expect
that necessarily when $M=I$ (i.e. no preconditioning is considered in
(\ref{prec_system})) $CD$ outperforms the
CG. Indeed, as stated in the previous section, $M=I$ along with
bounds (\ref{sigma_{k-2}}), (\ref{sigma_{k-1}}) and
(\ref{new_bound_beta}) do not suggest a specific preference for $CD$ with respect to the CG. On the
contrary, suppose a suitable preconditioner ${\cal M}= (M^T
M)^{-1}$ is selected when $\kappa (A)$ is large. Then, since the
class $CD$ for suitable values of $\gamma_{k-1}$ at Step $k$ possibly imposes stronger conjugacy conditions
with respect to the CG, it may possibly better recover the
conjugacy loss.
\\
We will soon see that if the preconditioner ${\cal M}$ is adopted
in $CD$, it is just used throughout the computation
of the product ${\cal M} \times v$, $v \in \mathbb{R}^n$, i.e. it is not necessary to
store the possibly dense matrix ${\cal M}$.

The algorithms in $CD$ for (\ref{prec_system}) are described
in Table~\ref{tab_prec_CG-2step}, where each `bar' quantity has a
corresponding quantity in Table~\ref{tab_CG-2step}. Then, after
substituting in Table \ref{tab_prec_CG-2step}
the positions
\begin{equation}
\begin{array}{ccl}
        \bar y_k & := & M y_k  \\
        \bar p_k & := & M p_k  \\
        \bar r_k & := & M^{-T} r_k \\
        {\cal M} & := & \left(M^T M \right)^{-1},
\end{array}  \label{further_pos}
\end{equation}
the vector $\bar p_k$ becomes
    $$\bar p_k = M p_k = \bar \gamma_{k-1} M^{-T} A M^{-1} M p_{k-1} - \bar \sigma_{k-1} M p_{k-1} - \bar \omega_{k-1} M p_{k-2},$$
hence
    $$ p_k = \bar \gamma_{k-1}{\cal M} A p_{k-1} - \bar \sigma_{k-1} p_{k-1} - \bar \omega_{k-1} p_{k-2} $$
with
\begin{eqnarray}
        \bar \sigma_{k-1} & = & \bar \gamma_{k-1}\frac{\|M^{-T} A p_{k-1}\|^2}{p_{k-1}^T A p_{k-1}} = \bar \gamma_{k-1}\frac{(A p_{k-1})^T {\cal M} A p_{k-1}}{p_{k-1}^T A p_{k-1}}  \label{3_stars} \\
                \ & &  \nonumber  \\
        \bar \omega_{k-1} & = & \frac{\bar \gamma_{k-1}}{\bar \gamma_{k-2}}\frac{p_{k-1}^T M^T M^{-T} A M^{-1} M p_{k-1}}{p_{k-2}^T M^T M^{-T} A M^{-1} M p_{k-2}} = \frac{\bar \gamma_{k-1}}{\bar \gamma_{k-2}} \frac{p_{k-1}^T A p_{k-1}}{p_{k-2}^T A p_{k-2}}. \nonumber
\end{eqnarray}
Moreover, relation $\bar r_0 = \bar b - \bar A \bar y_0$ becomes
    $$ M^{-T} r_0 = M^{-T} b - M^{-T} A M^{-1} M y_0 \qquad \Longleftrightarrow \qquad r_0= b-Ay_0, $$
and since $\bar p_0 = M p_0 = \bar r_0 = M^{-T}r_0$ then $p_0={\cal M} r_0$, so that the coefficients $\bar \sigma_0$ and $\bar a_0$ become
\begin{eqnarray}
        \bar \sigma_0 & = & \bar \gamma_{0} \frac{p_0^T M^T M^{-T} A M^{-1} M^{-T} A M^{-1} M p_0}{p_0^T A p_0} = \bar \gamma_{0} \frac{(Ap_0)^T {\cal M} (A p_0)}{p_0^T A p_0} \nonumber   \\
                      & = & \bar \gamma_{0} \frac{\|Ap_0\|_{\cal M}^2}{p_0^T A p_0}   \label{6_stars}  \\
        \bar a_0 & = & \frac{r_0^T M^{-1} M p_0}{p_0^T M^T M^{-T} A M^{-1} M p_0} = \frac{r_0^T p_0}{p_0^T A p_0}. \nonumber
\end{eqnarray}
As regards relation $\bar p_1 = \bar \gamma_{0} \bar A \bar p_0 - \bar \sigma_0 \bar p_0$ we have
    $$ M p_1 = \bar \gamma_{0} M^{-T} A M^{-1} M p_0 - \bar \sigma_0 M p_0, $$
hence
    $$p_1 = \bar \gamma_{0} {\cal M}A p_0 - \bar \sigma_0 p_0.$$
Finally, $\bar r_k = M^{-T} r_k$ so that
    $$ \bar r_k = M^{-T} r_k = M^{-T} r_{k-1} - \bar a_{k-1} M^{-T} A M^{-1} M p_{k-1}$$
and therefore
    $$ r_k = r_{k-1} - \bar a_{k-1} A p_{k-1},$$
with
    $$ \bar a_{k-1} = \frac{r_{k-1}^T M^{-1} M p_{k-1}}{p_{k-1}^T M^T M^{-T} A M^{-1} M p_{k-1}} = \frac{r_{k-1}^Tp_{k-1}}{p_{k-1}^T A p_{k-1}}.$$
The overall resulting preconditioned algorithm $CD$$_{\cal M}$ is detailed in
Table~\ref{tab_prec_M_CG-2step}. Observe that the coefficients
$a_{k-1}$ and $\omega_{k-1}$ in Tables~\ref{tab_CG-2step} and
\ref{tab_prec_M_CG-2step} are invariant under the introduction of
the preconditioner ${\cal M}$. Also note that from (\ref{3_stars})
and (\ref{6_stars}) now in $CD$$_{\cal M}$
the coefficient $\sigma_{k-1}$  depends on $A {\cal M} A$
and not on $A^2$ (as in Table \ref{tab_CG-2step}). \\
Moreover, in Table~\ref{tab_prec_M_CG-2step} the introduction of
the preconditioner simply requires at Step $k$ the additional cost
of the product ${\cal M} \times (Ap_{k-1})$ (similarly to the
preconditioned CG, where at iteration $k$ the additional cost of
preconditioning is given by ${\cal M} \times r_{k-1}$).

Furthermore, in Table~\ref{tab_prec_M_CG-2step} at Step $0$ the
products ${\cal M}r_0$ and ${\cal M}(Ap_0)$ are both required, in
order to compute $\sigma_0$ and $a_0$. Considering that Step
0 of $CD$ is equivalent to two iterations of the CG,
then the cost of preconditioning either CG or $CD$ is
the same. Finally, similar results hold if $CD$$_{\cal M}$ is recast in view of Remark \ref{rem:1}.

\section{Conclusions}
\label{sec:conclusions}
\begin{figure*}
\begin{center}
\includegraphics[width=0.66\textwidth]{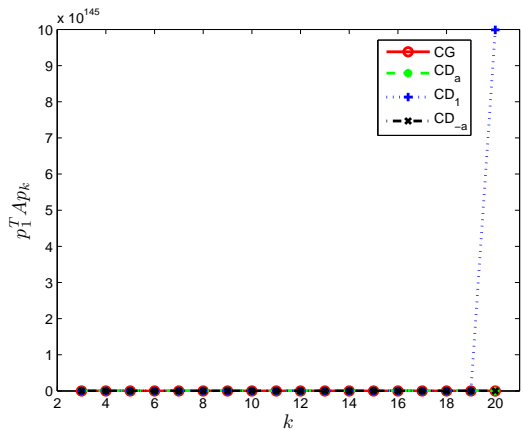}
\end{center}
\caption{Conjugacy loss for an illconditioned problem described by the coefficient matrix $A_{1_0}$ in \cite{MN00}, using the CG, $CD_a$ (the $CD$ class setting $\gamma_0=1$ and $\gamma_k=a_k$, $k\geq 1$), $CD_1$ (the $CD$ class setting $\gamma_k=1$, $k\geq 0$) and $CD_{-a}$ (the $CD$ class setting $\gamma_0=1$ and $\gamma_k=-a_k$, $k\geq 1$). The quantity $p_1^TAp_k$ is reported for $k\geq 3$. As evident, the choice $\gamma_k=1$, $k\geq 0$, can yield very harmful results when the coefficient matrix is illconditioned}
\label{fig:comp_1}       
\end{figure*}
\begin{figure*}
\begin{center}
\includegraphics[width=0.76\textwidth]{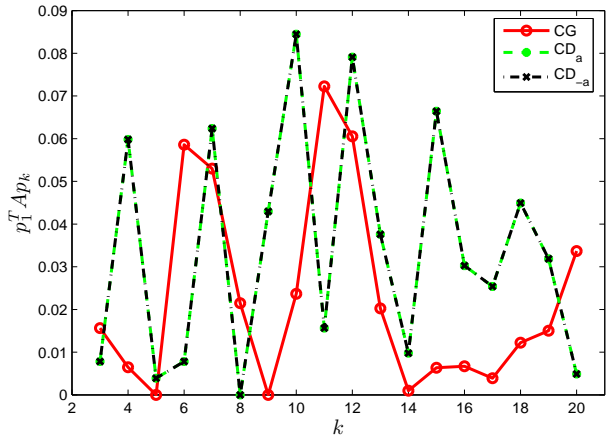}
\end{center}
\caption{Conjugacy loss for an illconditioned problem described by the coefficient matrix $A_{1_0}$ in \cite{MN00}, using only the CG, $CD_a$ (the $CD$ class setting $\gamma_0=1$ and $\gamma_k=a_k$, $k\geq 1$) and $CD_{-a}$ (the $CD$ class setting $\gamma_0=1$ and $\gamma_k=-a_k$, $k\geq 1$). The quantity $p_1^TAp_k$ is reported for $k\geq 3$. The choices $\gamma_k=a_k$ and $\gamma_k=-a_k$ are definitely comparable, and are preferable to the CG for $k \in \{3,6,8,11,20\}$.}
\label{fig:comp_2}       
\end{figure*}
We have investigated a novel class of CG-based iterative methods. This allowed us to recast several properties of the CG within a broad framework of iterative methods, based on generating mutually conjugate directions. Both the analytical properties and the geometric insight where fruitfully exploited, showing that general CG-based methods, including the CG and the scaled-CG, may be introduced. Our  resulting parameter dependent CG-based framework has the distinguishing feature of including conjugacy in a more general fashion, so that numerical results may strongly rely on the choice of a set of parameters. We urge to recall that in principle, since conjugacy can be generalized to the case of $A$ indefinite (see for instance \cite{H80,F05,FaRo07,F05a}) potentially further generalizations with respect to $CD$ can be conceived (allowing the matrix $A$ in (\ref{system}) to be possibly indefinite).
\\
Our study and the present conclusions are not primarily inspired by the aim of possibly beating the performance of the CG on practical cases. On the contrary, we preferred to justify our proposal in the light of a general analysis, which in case (but not necessary) may suggest competitive new iterative algorithms, for solving positive definite linear systems. In a future work we are committed to consider the following couple of issues:
\begin{enumerate}
    \item assessing clear rules for the choice of the sequence $\{\gamma_k\}$ in $CD$;
    \item performing an extensive numerical experience, where different choices of the parameters $\{\gamma_k\}$ in our framework are considered, and practical guidelines for new efficient methods might be investigated.
\end{enumerate}
The theory in Sects.~\ref{sec:further_prop} - \ref{sec:conjugacy_loss} seems to provide yet premature criteria, for a fruitful choice of the sequence $\{\gamma_k\}$ on applications. Furthermore, we do not have clear ideas about the real importance of the scheme $CD$-red in Table \ref{tab_CG-red}, where the choice (\ref{equ:relaz_gamma}) is privileged.
Anyway, to suggest the reader some numerical clues about our proposal, consider that the apparently simplest choice $\gamma_k=1$, $k\geq 0$, proved to be much {\em inefficient} in practice, while the choices $\gamma_k= \pm a_k$ gave appreciable results on different test problems (but still unclear results on larger test sets).

In particular we preliminarily tested the $CD$ class on two (small but) illconditioned problems described in Section 4 of \cite{MN00}. The first problem, whose coefficient matrix is addressed as $A_{1_0} \in \mathbb{R}^{50 \times 50}$, is `{\em obtained from a one-dimensional model, consisting of a line of two-node elements with support conditions at both ends and a linearly varying body force}'. The second  problem has the coefficient matrix $A_{2_0} \in \mathbb{R}^{170 \times 170}$, which is `{\em the stiffness matrix from a two-dimensional
finite element model of a cantilever beam}'.
\\
In Figures \ref{fig:comp_1}-\ref{fig:comp_2} we report the resulting experience on just the first of the two problems (similar results hold for the other one), where the CG is compared with algorithms in the class $CD$, setting $\gamma_k \in \{a_k,1,-a_k\}$. As a partial justification for the reported numerical experience, we note that in the $CD$ class the coefficient $\sigma_k$ depends on the quantity $\|Ap_k\|^2$. Thus, $\|Ap_k\|^2$ may be large when $A$ is illconditioned, so that the choice $\gamma_k=1$ possibly is inadequate to compensate the effect of illconditioning. On the other hand, setting $\gamma_k=\pm a$ and considering the expression of $a_k$, the coefficient $\sigma_k$ is possibly re-scaled, taking into account the condition number of matrix $A$.

Observe that the algorithms in $CD$ are slightly more expensive
than the CG, and they require the storage of one further
vector with respect to the CG. However, we proved for $CD$ some
theoretical properties, which extend those provided by the CG, in order to possibly prevent from conjugacy loss. In addition, when specific values of the parameters in $CD$ are chosen, then we obtain schemes equivalent to both the CG and the scaled-CG.
\\
Furthermore, we have also introduced preconditioning in our proposal, as a possible extension of the preconditioned CG, so that illconditioned linear systems might be possibly more efficiently tackled. Our methods are also aimed to provide an effective tool in optimization contexts where a sequence of conjugate
directions is sought. Truncated Newton methods are just an example of such contexts
from unconstrained nonlinear optimization, as detailed in Sect.~\ref{sec:2b}. We are considering in a further study a numerical experience, over
convex optimization problems, where $CD$ and the relative preconditioned scheme are adopted to solve Newton's equation.
Indeed, in case the matrix $A$ in (\ref{system}) is indefinite, the choices $\gamma_k \in \{a_k , |a_k| ,
-a_k , - |a_k|\}$ are of some interest and might be compared on a significant test set.

In addition, it might be worth also to investigate the choice where the preconditioner ${\cal M}$ in Table \ref{tab_prec_M_CG-2step} is computed by a Quasi-Newton approximation of the inverse matrix $A^{-1}$ (see also \cite{MN00,Gratton:09}), or by using the conjugate directions generated by $CD$, for a suitable choice of the parameters (see also \cite{FaRo13}).

Furthermore, observe that conditions (\ref{two_star}) or (\ref{two_star_alternative}) cannot be further generalized imposing explicitly relations ($\ell \geq 1$)
    $$ p_k^TAp_j=0, \qquad j=k-1,k-2,  \ldots, k-\ell, $$
since (\ref{two_star}) and (\ref{two_star_alternative}) automatically imply $p_k^TAp_j=0$, for any $j\leq k-3$ (see also Lemma \ref{lemma0} and Lemma \ref{lemma2}).

Finally, note that for the minimization of a convex quadratic functional in $\mathbb{R}^n$, the complete relation between the search directions generated by BFGS or L-BFGS updates and the CG was studied (see also \cite{NW00}). Thus, we think that  possible extensions may be considered by replacing the CG with the algorithms in our framework. In this regard, recalling that polarity (see \cite{H80}) plays a keynote role for generating conjugate directions, there is the chance that a possible relation between the BFGS update and $CD$ could spot some light on the role of polarity for Quasi-Newton schemes.



\begin{acknowledgements}
The author is indebted with the anonymous reviewers and the Editor in Chief for their fruitful comments.
\end{acknowledgements}



\end{document}